\documentclass[11pt,a4paper]{amsart}
\usepackage{hyperref,texmac,enumitem,xifthen}
\usepackage[all]{xy}
\usepackage{amsmath,mathtools}
\xyoption{web}

\operators{Aut,Div,Pic,Pid,Id,Cl,ab,Norm,cts,Ext,opp,J,M,skew,com,Core,Stab,Epi,sign}
\operators{Ara,fin,im,torsfree,c,Nm,tori,D,f,ring,CL,Surj,Av,PSL,cpt}
\calsymbols{c}{O,C,B,A,R,K,M,D,S,F,X,Y,L,P,G,E}
\bbsymbols{b}{E,A,P,T,H}
\fraksymbols{f}{F,p,A,M,m,a,n}

\DeclareMathOperator{\mymod}{mod}
\DeclareMathOperator{\myfunc}{F}
\DeclareMathOperator{\mydiag}{D}
\DeclareMathOperator{\myseq}{S}

\def\AbGp{{\bf AbGp}}
\def\CAbGp{{\bf CAbGp}}
\def\CCAbGp{{\bf CCAbGp}}
\def\PAbGp{{\bf PAbGp}}
\def\TAbGp{{\bf TAbGp}}
\def\TFAbGp{{\bf TFAbGp}}
\newcommand{\AbsNm}[1][]{%
\ifthenelse{\equal{#1}{}}{|\!\Nm\!|}{|\!\Nm\!|(#1)}%
}
\def\bfa{\mathbf{a}}
\def\pairs{\Sigma}
\def\pairsmap{\sigma}
\def\raytooriented{t}
\def\FRtoC{\rho}
\def\ourisometry{r}
\def\sop{S} %set of primes

\begin{document}
\title[Ray class groups of quadratic fields]
{A heuristic for ray class groups of quadratic number fields}
\author{Alex Bartel}
\address{School of Mathematics and Statistics, University of Glasgow,
University Place, Glasgow G12 8SQ, United Kingdom}
\author{Carlo Pagano}
\address{Department of Mathematics, Concordia University, Montreal H3G 1M8, Canada}

\email{alex.bartel@glasgow.ac.uk, carlo.pagano@concordia.ca}

\begin{abstract}
  We formulate a model for the average behaviour of ray class groups
  of real quadratic fields with respect to a fixed rational modulus,
  locally at a finite set $S$ of odd primes.
  To that end, we introduce Arakelov ray class groups of a number
  field, and postulate that, locally at $S$,
  the Arakelov ray class groups of real quadratic fields are distributed
  randomly with respect to a natural Cohen--Lenstra type probability measure.
  We show that our heuristics imply the Cohen--Lenstra heuristics on class
  groups of real quadratic fields, as well as equidistribution results
  on the fundamental unit of a real quadratic field modulo an integer,
  and are consistent with Varma's results on average of sizes of $3$-torsion
  subgroups of ray class groups of quardratic fields.
\end{abstract}
\date{\today}

\maketitle
\tableofcontents
\section{Introduction}
\subsection{Ray class groups}\label{sec:ray}
In \cite{Varma} Varma determined the average size of the $3$-torsion subgroups
of the ray class groups of imaginary and of real quadratic fields, ordered
by absolute discriminant, with respect to a fixed rational modulus. This
raised the natural problem of finding a model that ``explains'' Varma's formulae,
analogously to the way in which the Cohen--Lenstra heuristics \cite{CL} explained
the result of Davenport--Heilbronn \cite{DH} on the average size of $3$-torsion of class
groups of quadratic fields. Bhargava had already publicised this problem at an AIM workshop
back in 2011. The second author and Sofos have found such a
model for ray class groups of imaginary quadratic fields \cite{PS}.
In this paper we do the same for the family of real quadratic fields.

%We formulate several natural
%Cohen--Lenstra type heuristics for ray class groups of real quadratic number fields,
%prove that they are equivalent, and provide some evidence for them: we show that 
%our heuristics specialise to the Cohen--Lenstra--Martinet heuristics on ideal class 
%groups of number fields, and we verify that they are compatible with the results of 
%Varma \cite{Varma} on $3$-torsion of ray class groups of quadratic fields.
A central r\^ole in our work is played by Arakelov ray class groups of a number field.
It is a beautiful object that does not appear to have been 
treated in the literature, and we devote some space to shining a spotlight on it 
and investigating its properties. As a byproduct, in Section \ref{sec:geodesics}
we reinterpret a classical construction that assigns to each narrow ideal class of a real
quadratic field a modular geodesic: we upgrade this construction to a collection of
isometries from narrow Arakelov class groups, whose connected components
are narrow ideal classes, to the quotient of the hyperbolic upper half plane
by its isometry group $\PSL_2(\Z)$.

Let $F$ be a number field, let $\cO_F$ denote its ring of integers, and let
$\fm=(\fm_{\f},\fm_{\infty})$ be a modulus, where $\fm_{\f}$ is an ideal of $\cO_F$ and
$\fm_{\infty}$ is a set of real embeddings of $F$. Let $F^1(\fm)$ be the set
of all $\alpha\in F^\times$ satisfying $\sigma(\alpha)>0$ for all
$\sigma\in \fm_{\infty}$ and $\ord_{\fp}(\alpha-1)\geq \ord_{\fp}(\fm_{\f})$
for all prime ideals $\fp$ of $\cO_F$ dividing $\fm_{\f}$, where $\ord_{\fp}$
denotes the $\fp$-adic valuation. The ray class group
$\Cl_F(\fm)$ of $F$ with modulus $\fm$ is the quotient of the group $\Id_F(\fm)$ of
all fractional ideals of $\cO_F$ that are coprime to $\fm_{\f}$ by the subgroup of
all principal fractional ideals of the form $(\alpha)$ for some
$\alpha\in F^1(\fm)$. The class group $\Cl_F$ is 
the ray class group $\Cl_F((\cO_F,\{\}))$. One has a canonical short exact sequence
\begin{eqnarray}\label{eq:SeqRay}
  \myseq_F^{\fin}(\fm)\colon 0\to
  \frac{(\cO_F/\fm_{\f})^\times\times \{\pm1\}^{\fm_{\infty}}}{(\cO_F^\times)_{\mymod \fm}}
    \to \Cl_F(\fm) \to \Cl_F \to 0,
\end{eqnarray}
where $(\cO_F^\times)_{\mymod \fm}$ denotes the image of $\cO_F^\times$ in
$(\cO_F/\fm_{\f})^\times\times\{\pm1\}^{\fm_{\infty}}$ under the canonical map given
by the quotient map on the first factor, and by
$\alpha\mapsto (\sign(\sigma(\alpha))_{\sigma\in \fm_{\infty}}$
on the second factor.
%
%In \cite{CL} Cohen and Lenstra formulated probabilistic models for the behaviour
%of $\Cl_F$ as $F$ runs over imaginary, respectively real quadratic number fields.
%It was extended by Cohen and Martinet in \cite{CM} to more general families
%of number fields. In \cite{PS} the second author and Sofos formulated a model for 
%$\Cl_F(\fm)$ when $\fm_{\f}$ is generated by a fixed rational integer and $F$ runs over imaginary quadratic 
%number fields, and proved several theorems to support that model.
%In the present paper we do the same for the family of real quadratic fields (away from the prime $2$ --
%see Section \ref{sec:mainheuristic} below).

There is a well-established principle that a ``random'' algebraic object 
is isomorphic to a given object $X$ with probability propotional to $1/\#\Aut(X)$. 
It forms the starting point of the Cohen--Lenstra heuristics, where it is conjectured 
to apply to ideal class groups of imaginary quadratic fields \cite{CL}. The first
author and Lenstra showed in \cite{BL} that in order for this principle to
apply to class groups of general number fields, one should instead consider so-called
Arakelov class groups. We will show in Theorem \ref{thm:realquad} below that the
conjectures in \cite{PS} on ray class groups of imaginary quadratic number fields can
be understood as saying that, roughly speaking, the exact sequence $\myseq_F^{\fin}(\fm)$
is isomorphic to a given suitable exact sequence $E$ with probaility proportional
to $1/\#\Aut(E)$. In order to make the same principle applicable to ray class groups of
real quadratic fields, we introduce the so-called Arakelov ray class group of a number field.

%It therefore seems plausible to model $\Cl_F(c)$ as representing a uniformly
%random extension class of a ``Cohen--Lenstra random'' group (namely $\Cl_F$) by
%a random group that is obtained as the quotient of $\frac{(\cO_F/c)^\times}{(\mu_F)_{\mymod c}}$
%(which may be regarded as constant in judiciously chosen families) by the subgroup
%generated by $\rk_{\Z} \cO^\times_F$ uniformly independently random elements. Once
%one makes all of this precise, and corrects for numerous technicalities, we shall
%see that that is indeed what our model will predict. However, as we will explain
%in the next subsection, by developing the entire model in terms of Arakelov ray
%class groups, we will instead show that all the above mentioned equidistribution
%hypotheses become corollaries of the very natural principle that a random algebraic
%object is isomorphic to a given object $X$ with probability propotional to $1/\#\Aut(X)$.
%That principle has been well known for a long time, and has been the main motivation
%behind the original Cohen--Lenstra heuristics. The main challenge that we address here
%is to find the right algebraic object to which that principle applies in the context
%of ray class groups.

\subsection{Arakelov ray class groups}\label{sec:introgoodprimes}
We refer to Section \ref{sec:defs} for details that we leave out for now.
For a number field $F$, we define the following invariants attached to $F$.
Let $\overline{F_{\R}^\times}$ denote the quotient of $(F\otimes_{\Q}\R)^\times$
by its maximal compact subgroup.
The group $\overline{F_{\R}^\times}$ is isomorphic
to $\prod_{v\in S_{\infty}} \R_{>0}$, where $S_{\infty}$ denotes the set of
infinite places of $F$. We have natural maps $\Id_F(\fm)\to \R_{>0}$ and
$\overline{F_{\R}^\times}\to \R_{>0}$, the first induced by the ideal norm,
and the second by the absolute value of the $\R$-algebra norm on $F\otimes_{\Q}\R$.
It follows from the product formula that the image of
the natural map $F^1(\fm)\to \Id_F(\fm)\times\overline{F_{\R}^\times}$ is contained in the fibre
product $\Id_F(\fm)\times_{\R_{>0}}\overline{F_{\R}^\times}$, and we define the
\emph{Arakelov ray class group} $\Pic_F^0(\fm)$ of $F$ with modulus $\fm$ to be the quotient
of $\Id_F(\fm)\times_{\R_{>0}}\overline{F_{\R}^\times}$ by the image of $F^1(\fm)$.
As we will explain in more detail in Section \ref{sec:defs}, this is a compact abelian group.
When we have $\fm=(\cO_F,\{\})$, we will drop it from the notation, and we refer to $\Pic_F^0$ as
the \emph{Arakelov class group} of $F$.
The natural map $\Id_F(\fm)\times_{\R_{>0}}\overline{F_{\R}^\times}\to
\Id_F\times_{\R_{>0}}\overline{F_{\R}^\times}$ induces a surjective map
$\Pic_F^0(\fm)\to \Pic_F^0$, which is the right hand map in the short exact sequence

\begin{eqnarray}\label{eq:SeqRayArak}
  \quad\quad \myseq^{\Ara}_F(\fm)\colon 0 \to \frac{(\cO_F/\fm_{\f})^\times\times\{\pm1\}^{\fm_{\infty}}}{(\mu_F)_{\mymod \fm}} \to \Pic_F^0(\fm) \to \Pic_F^0 \to 0.
\end{eqnarray}

The great advantage of that exact sequence compared to the sequence
$\myseq^{\fin}_F(\fm)$ above is that if $\fm_{\f}$ is
generated by a fixed rational integer and $\fm_{\infty}$ is empty,
then the dependence of the left hand term of $\myseq^{\Ara}_F(\fm)$ on $F$ is easy to
control: it only depends on $\mu_F$ and
on the splitting behaviour of $\fm_{\f}$ in $F$. As a consequence, under these
hypotheses, any family of number fields that one considers in Cohen--Lenstra
type heuristics can be partitioned into finitely many subfamilies with easy to
compute densities such that the left hand term in $\myseq^{\Ara}_F(\fm)$ is 
constant in each subfamily.

%We will also need to consider the dual exact sequence $\myseq^{\Ara}_F(\fm)^\vee$,
%obtained from $\myseq^{\Ara}_F(\fm)$ by applying the Pontryagin duality functor
%$\Hom_{\cts}(-,\R/\Z)$.

\begin{assumption}\label{ass}
For the rest of the introduction, assume that $\fm_{\infty}$ is empty and that
$\fm_{\f}$ is generated by a rational integer.
\end{assumption}
Under this assumption, it makes
sense to ask how $\Pic_F^0(\fm)$ behaves ``on average'' as $F$ varies over natural
families of number fields. We will now formulate our main heuristic concerning
that behaviour and state the main theorems for the family of real quadratic fields.

\subsection{The main heuristic}\label{sec:mainheuristic}
Let $C_2$ be a group of order $2$. Let $\sop$ be a finite set of odd prime numbers,
let
$$
\Z_{\sop}=\prod_{p\in \sop}\Z_p
$$
be the ring of $\sop$-adic integers.
%\{a/b: a,b\in \Z, b\not\in \bigcup_{p\in S\cup\{0\}}p\Z\}$.
We call a torsion abelian group an \emph{$\sop$-group} if the order of every element is some
product of primes in $\sop$. If $C$ is an abelian group,
abbreviate $\Z_{\sop}\otimes_{\Z}C$ to $C_\sop$, and denote by $C(\sop^\infty)$
the $\sop^\infty$-torsion subgroup of $C$, i.e. the largest subgroup
of $C$ that is an $\sop$-group. Let $\cG^-_{\cpt}$ be a full set
of representatives of isomorphism classes of finite abelian groups,
and let $\cG^+_{\cpt}$ be a full set of representatives of isomorphism classes of
compact abelian groups whose connected component of the identity is isomorphic to
$\R/\Z$. 
Let $\cG^-_S$ be a full set of representatives
of isomorphism classes of finite abelian $\sop$-groups, and let $\cG^+_{\sop}$ be a
full set of representatives of isomorphism classes of abelian $\sop$-groups
whose maximal $\sop$-divisible subgroup is isomorphic to $\prod_{p\in\sop}\Q_p/\Z_p$
and whose quotient by that subgroup is finite. We endow all of these groups with the structure of a $C_2$-module,
with the non-trivial element acting by multiplication by $-1$.
Thus, for every imaginary, respectively real quadratic field $F$, the Arakelov class group $\Pic_F^0$
is isomorphic, as a $\Z[C_2]$-module via the Galois action, to a unique element of $\cG^-_{\cpt}$, respectively
$\cG^+_{\cpt}$, while $\Pic_F^0(\sop^\infty)$ is isomorphic to a unique element of $\cG^-_{\sop}$,
respectively~$\cG^+_{\sop}$.

%In both cases, by a ``full set of representatives of isomorphism classes'' we mean
%that for every such group $C'$, there exists a unique $C\in \cG^-$, respectively $\cG^+$ that is 
%isomorphic to $C'$. The elements of the sets $\cG^-$ and $\cG^+$ will be models
%for the $S^\infty$-torsion subgroups of the Arakelov class groups of imaginary, respectively real
%quadratic fields.
%
We will use the sign $\pm$ to mean
``one of $+$ or $-$''. We will shortly define other structures that depend on this
sign, and there, $\pm$ will always be understood to denote the appropriate sign.

%We regard every element of $\cG^{\pm}$ as a $C_2$-module with the 
%non-trivial element acting as multiplication by $-1$.

Let $\fm_{\f}\in \Z$. We will now set up the partition of the families of quadratic number fields $F$
according to the splitting behaviour of $\fm_{\f}$ in $F$, as mentioned at the end of the previous
subsection. Slightly abusing notation, we will also regard $\fm_{\f}$ as an ideal in
varying rings of integers $\cO_F$. For every prime divisor $p$ of $\fm_{\f}$, fix a
degree $2$ \'etale $\Q_p$-algebra $A_p$, viewed as a $C_2$-module by identifying
$C_2$ with the automorphism group of $A_p/\Q_p$, let $A=\prod_{p|\fm_{\f}}A_p$ with the product 
taken over prime divisors, let $\cO_A$ be the integral closure of $\prod_{p|\fm_{\f}}\Z_p$
in $A$, and let $R=\cO_A/\fm_{\f}$. Finally, set $U_R=R^{\times}/\{\pm 1\}$.
%Since we regard $\sop$ as being fixed once and for all, we suppress the dependence on
%$\sop$ from the notation.
If $F$ is a real, respectively imaginary quadratic number field other than $\Q(\sqrt{-3})$,
then the exact sequence \eqref{eq:SeqRayArak} gives rise, upon taking the quotient of $U_R$
by its maximal subgroup of order coprime to $S$, to a short exact sequence of $\Z[C_2]$-modules
of the form
\begin{eqnarray}\label{eq:theta}
\Theta\colon 0\to U_R(S^\infty)\to B \to C\to 0,
\end{eqnarray}
where $C\in \cG^{\pm}_{\cpt}$. Correspondingly, as a model for the 
$\sop^\infty$-torsion of the exact sequence \eqref{eq:SeqRayArak} consider short
exact sequences
\begin{eqnarray}\label{eq:thetaS}
\Theta\colon 0\to U_R(\sop^\infty)\to B \to C\to 0
\end{eqnarray}
with $C\in \cG^{\pm}_{\sop}$.
%$(C,i)$, where $C$ is an abelian
%real Lie group with a continuous $C_2$-action, and $i\colon (R^{\times}/\{\pm 1\})_{(S)}\hookrightarrow C$ is an 
%embedding of $C_2$-modules with the property that the quotient by its image $\im i$
%is isomorphic as a topological $C_2$-module to an element 
%of $\cG^{\pm}$.
Let 
\begin{align*}
  \Theta\colon & 0\to U_R(\sop^\infty)\stackrel{i}{\to} B \to C\to 0,\\
  \Theta'\colon & 0\to U_R(\sop^\infty)\stackrel{i'}{\to} B' \to C'\to 0
\end{align*}
be two sequences either both as in \eqref{eq:theta} or both as in \eqref{eq:thetaS}.
We define an isomorphism $\Theta\to \Theta'$ to be a $C_2$-module
isomorphism $f\colon B\to B'$ that sends $\im i$ isomorphically to
$\im i'$ and such that the automorphism $(i')^{-1}\circ f\circ i$ of $U_R(\sop^\infty)$
is induced by a ring automorphism of $R$. So as to stress this last condition,
we will write $\Aut_{\ring}(\Theta)$ for the group of such automorphims of
$\Theta$.
Let $\cG_{R,\cpt}^{\pm}$, respectively $\cG_{R,\sop}^{\pm}$ be a full set of representatives
of isomorphism classes of sequences as in \eqref{eq:theta}, respectively as in~\eqref{eq:thetaS}.

It is easy to see, and will in particular follow from Proposition
\ref{prop:sizeaut}, that every element
of $\cG_{R,\cpt}^{\pm}$ has finite automorphism group. Moreover, we will show in 
Corollary \ref{cor:converge},
that the two sums $\sum_{\Theta\in \cG_{R,\cpt}^{\pm}}1/\#\Aut_{\ring}\Theta$ both converge,
to $c_R^{\pm}$, say.
We define discrete probability distributions $\bP_{\cpt}^{\pm}$ on
$\cG_{R,\cpt}^{\pm}$ by 
$$
\bP_{\cpt}^{\pm}\{\Theta\} = \frac{(c_R^{\pm})^{-1}}{\#\Aut_{\ring}\Theta}.
$$
There are obvious functions $\cG_{R,\cpt}^{\pm}\to \cG_{R,\sop}^{\pm}$ that
send an element $\Theta\in \cG_{R,\cpt}^{\pm}$ to the unique element
of $\cG_{R,\sop}^{\pm}$ that $\Theta(S^\infty)$ is isomorphic to. Let
$\bP^{\pm}$ be the pushfowards of the probability distributions $\bP^{\pm}_{\cpt}$
under these functions.

If $f$ is a complex-valued $L^1$-function on one of the two probability spaces
$(\cG_{R,\sop}^{\pm},\bP^{\pm})$, then we denote its expected
value by $\bE^{\pm}(f)$, and for every sequence $\Theta$ as above we define $f(\Theta)$ to be the
value of $f$ on the unique element of $\cG_{R,\sop}^{\pm}$ that is isomorphic to $\Theta$.

For a real number $X>0$, let $\cF_X^-(R)$ respectively $\cF_X^+(R)$ be the set
of pairs $(F,r)$, where $F$ is an imaginary, respectively real quadratic number
field of conductor less than $X$, and $r$ is a $C_2$-equivariant ring isomorphism
$r\colon R\to \cO_F/\fm_{\f}$. For every $X\in \R_{>0}$ and for every
$(F,r)\in \cF_X^{\pm}(R)$ except for $F\cong \Q(\sqrt{-1})$ and $F\cong \Q(\sqrt{-3})$,
the $\sop^\infty$-torsion $\myseq_F^{\Ara}(\fm)(\sop^\infty)$ of the exact sequence
$\myseq_F^{\Ara}(\fm)$ defined in (\ref{eq:SeqRayArak})
%the pair consisting of the ray class group $\Cl_F(\fm)$ and of the canonical embedding
%$(\cO_F/\fm_{\f})^\times/\{\pm1\}\hookrightarrow \Cl_F(\fm)$
is isomorphic, via $r$, to a unique element of $\cG_{R,\sop}^{\pm}$, where
the $C_2$-action is induced by the Galois action on $F$.
Our main heuristic assumption reads as follows.

\begin{heuristic}\label{he:realquad}
  For $X\in \R_{>0}$, let $\cF_X^{\pm}(R)$, $\cG_{R,\sop}^{\pm}$, and $\bE^{\pm}$ be as just defined.
  Then for every ``reasonable'' $\C$-valued function $f$ on $\cG_{R,\sop}^{\pm}$, the
  average
  $$
  \lim_{X\to \infty} \frac{\sum_{(F,r)\in \cF_X^{\pm}(R)}f(\myseq_F^{\Ara}(\fm)(\sop^\infty))}{\#\cF_X^{\pm}(R)}
  $$
  exists, and is equal to $\bE^{\pm}(f)$.
\end{heuristic}

\begin{remarks}
  \begin{enumerate}[leftmargin=*]
    \item The heuristic can be intuitively understood as saying that
      for a ``random'' real, respectively imaginary quadratic field $F$,
      the $\sop^\infty$-torsion of the exact sequence $S^{\Ara}_F(\fm)$
      defined in (\ref{eq:SeqRayArak}) is isomorphic to a given 
      element of $\cG_{R,\sop}^{\pm}$ with a probability that is inversely proportional to the 
      number of automorphisms of the latter.
      %Thus the exact sequence $S^{\Ara}_F(c)$
      %conforms to the general Cohen--Lenstra principle dicussed above.
    \item See \cite[\S 7]{BL} for a detailed discussion of what functions might
      be regarded as ``reasonable'' in the context of Cohen--Lenstra heuristics. 
      Here we will not concern ourselves further with this imprecision, but will
      just suggest that it does not seem inconceivable that functions that are
      $L^i$ for all $i\geq 1$ with respect to $\bP^{\pm}$ are all ``reasonable''.
%    \item In the generalisation of Heuristic \ref{he:realquad} to other families
%      of number fields, which will be formulated in Heuristic \ref{he:general} below,
%      an important adjustment is that Arakelov ray class
%      groups of general number fields should not be modelled as topological groups,
%      but as Galois modules, similar to the adjustment introduced by Cohen--Martinet
%      \cite{CM} in the case of ideal class groups. The main technical problem that
%      arises in this case is that the automorphism group of the exact sequence $\cG^{\Ara}_F(c)$
%      need no longer be finite, in general. This problem presents itself already in the
%      case of Arakelov class groups, and will be solved by means of the theory of
%      commensurability of automorphism groups, introduced in \cite{commensurability} for that purpose.
    \item The finiteness condition on the set $\sop$ can likely be dropped in the case
      of real quadratic fields, but will be important for general families of number fields.
      As was shown in \cite[\S 4]{BL}, already the original Cohen--Lenstra--Martinet 
      heuristic is false without that assumption, in general.
%      It seems likely that 
%      this restriction can be loosened using the ideas that are proposed in \cite
%      {BJL} in the case of Arakelov class groups, but we do not pursue this avenue here.
    \item In Heuristic \ref{he:realquad} we enumerate real quadratic fields by conductor.
      One could equivalently enumerate them by discriminant, which is traditionally 
      by far the most frequently used invariant to enumerate number fields. However,
      it was shown in \cite[\S 6]{BL} that for more general families of number 
      fields, the Cohen--Lenstra--Martinet heuristic is false when number fields 
      are enumerated by discriminant.
%      In Heuristic \ref{he:general} below we follow 
%      \cite[\S 6]{BL} and propose to use the product of ramified primes, or their 
%      ideal norm in the case of relative extensions, as a general replacement for conductor.
  \end{enumerate}
\end{remarks}

\subsection{Consequences for ray class groups}
The simple postulate that the sequence $\myseq^{\Ara}_F(\fm)(\sop^\infty)$ of the
$\sop^\infty$-torsion of the Arakelov ray class group is distributed with
probability weights that are inversely proportional to the number of
automorphisms leads to very satisfying equidistribution predictions for the
sequence $\myseq^{\fin}_F(\fm)(\sop^\infty)$ of the $\sop^\infty$-torsion
subgroup of the ray class group. We will formulate these consequences somewhat
sloppily here, and refer to Section \ref{sec:good} for precise formulations.
%let $\cG^{\fin}$ be a full set of representatives of isomorphism classes of
%finite abelian groups $C$ satisfying $C=C_{(S)}$, equipped with the discrete
%sigma algebra, and let $\bP^{\fin}_1$ be the unique probability distribution
%on $\cG^{\fin}$ with the property that for every $C\in \cG^{\fin}$ one
%has $\bP^{\fin}_1(\{C\})\sim \frac{1}{\#C\#\Aut C}$. This is the Cohen--Lenstra
%distribution for real quadratic number fields.
\begin{theorem}\label{thm:realquad}
  Suppose that Heuristic \ref{he:realquad} holds for the functions $f$ as in
  Corollaries \ref{cor:PFclassgroup} -- \ref{cor:equidistrExt}. Let $g$ be the non-trivial element
  of $C_2$. Then the following assertions are true.
  \begin{enumerate}[leftmargin=*]
    \item\label{item:CL} The Cohen--Lenstra heuristic for $\sop$-Hall subgroups
      $(\Cl_F)(\sop^\infty)$ of the class groups of imaginary and of real
      quadratic number fields holds.
    \item\label{item:equihom} In the limit as $X\to \infty$, as $(F,r)$ runs
      over $\cF^+_X(R)$, the distribution of the image of the fundamental unit
      of $\cO_F$ in $\frac{\Z_{\sop}[C_2]}{(g+1)}\otimes_{\Z}\left(\frac{R^\times}{\{\pm1\}}\right)$
      approaches the uniform distribution.
    \item\label{item:indep} In the limit as $X\to \infty$, as $(F,r)$ runs
      over $\cF^+_X(R)$, the distribution of $\Cl_F(\sop^\infty)$ and that of
      the image of the fundamental unit in
      $\frac{\Z_{\sop}[C_2]}{(g+1)}\otimes_{\Z}\left(\frac{R^\times}{\{\pm1\}}\right)$
      are independent of each other.
    \item\label{item:imagext} Let $G$ be a finite abelian $\sop$-group, and let
      $\cF_X^-(R,G)$ be the set of triples $(F,r,\iota)$ where $(F,r)\in \cF_X^-(R)$
      %is such that $(\Cl_F)(\sop^\infty)$ is isomorphic to $G$,
      and $\iota$ is an isomorphism $(\Cl_F)(\sop^\infty)\to G$. Then, in the limit as
      $X\to \infty$, when $(F,r,\iota)$ varies over $\cF_X^-(R,G)$, the class
      of $\myseq_F^{\fin}(\fm)(\sop^\infty)$ in
      $\Ext^1_{\Z_{\sop}[C_2]}\!\big(G,U_R(\sop^\infty)\big)$
      is equidistributed, in other words \cite[Heuristic assumption 2.4]{PS} holds.
    \item\label{item:realext} Let $u\in U_R(\sop^\infty)$
      be such that the generator of $C_2$ acts on $u$ by $u\mapsto u^{-1}$,
      let $Q=U_R(\sop^\infty)/\langle u\rangle$,
      let $A$ be a finite abelian $\sop$-group, and let
      $\cF_X^+(R,A,Q)$ be the set of triples $(F,r,\iota)$ where $(F,r)\in \cF_X^+(R)$,
      %is such that $(\Cl_F)(\sop^\infty)$ is isomorphic to $A$,
      $\iota$ is an isomorphism $(\Cl_F)(\sop^\infty)\to A$, and
      $\left(\frac{(\cO_F/\fm_{\f})^\times}{(\cO_F^\times)_{\mymod \fm}}\right)\!(\sop^\infty)$ is
      identified by $r$ with $Q$. Then
      in the limit as $X\to \infty$, as $(F,r,\iota)$ runs over $\cF_X^+(R,A,Q)$,
      the class of $\myseq_F^{\fin}(\fm)(\sop^\infty)$ in $\Ext^1_{\Z_{\sop}[C_2]}(A,Q)$ is equidistributed.
  \end{enumerate}
\end{theorem}
%Note that for every real quadratic field $F$, the homomorphism group
%$\Hom(\cO_F^\times/\{\pm 1\},R^\times/\pm 1)$ is finite, and part
%(\ref{item:equihom}) of Theorem \ref{thm:realquad} can be reformulated
%as stating that as $(F,r)$ varies, the homomorphism from $\cO_F^\times/\{\pm 1\}$
%to $R^\times/\pm 1$ induced by reduction modulo $c$, composed with $r$,
%is equidistributed in the homomorphism group. In that reformulation,
%the statement generalises well to other number fields, where, once again,
%the Galois action needs to be taken into account.

It may seem quite surprising that while, as we mentioned above, the sequence
$\myseq^{\Ara}_F(\fm)$ is attractive for statistical purposes precisely \emph{because} 
its left hand term does not vary in families, the sequence ``knows'' the image of 
$\cO_F^\times$ in $(\cO_F/\fm_{\f})^\times$. We will investigate this phenomenon in
the general context of exact sequences of compact groups, without reference to
number fields, in Section \ref{sec:functors}.

\subsection{Notation}
If $C$ is a locally compact abelian group, then $C^\vee$ will denote its
Pontryagin dual $\Hom_{\cts}(C,\R/\Z)$.
Throughout the paper, $\sop$ will be a finite set of odd prime numbers. 
If $C$ is an abelian group, then
$C(\sop^\infty)$ denotes the maximal $S^\infty$-torsion subgroup of $C$,
i.e. the subgroup generated by all elements whose orders are
products of primes in $\sop$; and we abbreviate $\Z_{\sop}\otimes_{\Z}C$ to $C_\sop$.
We always give $C(\sop^\infty)$ the discrete topology, not the subspace topology from $C$.
Note that for every compact Hausdorff abelian group $C$, we have a canonical isomorphism
$(C^\vee)_{\sop} = (C(\sop^\infty))^\vee$. Moreover, if $C$ is a finite abelian group,
then the quotient map $C\to C_\sop$ induces a canonical isomorphism $C(\sop^\infty)\cong C_\sop$.

For the convenience of the reader, we collect the notation that will remain in
force throughout the paper. Unless otherwise stated, $F$ will denote an arbitrary
number field.
\newpage
\begin{tabular}{l|l}
  $\Cl_F(\fm)$ & the ray class group of $F$ with modulus $\fm$\\
  $\Cl_F$ & $\Cl_F(\cO_F,\{\})$, the ideal class group of $F$\\
  $\Pic^0_F(\fm)$ & the Arakelov ray class group of $F$ with modulus $\fm$,\\
                & as defined in \S \ref{sec:introgoodprimes}\\
  $\Pic^0_F$ & $\Pic^0_F(\cO_F,\{\})$, the Arakelov class group of $F$\\
  $\myseq_F^{\fin}(\fm)$ & the ray class group exact sequence (\ref{eq:SeqRay})\\
  $\myseq_F^{\Ara}(\fm)$ & the Arakelov ray class group exact sequence (\ref{eq:SeqRayArak})\\
  $\cG_{\cpt}^-$ & full set of isomorphism class representatives of finite groups,\\
                 & with $C_2$ acting by $-1$\\
  $\cG_{\cpt}^+$ & full set of isomorphism class representatives of groups of the\\
                 & form $C\oplus \R/\Z$ for $C\in \cG_{\cpt}^-$, with $C_2$ acting by $-1$\\
  $\cG_{\sop}^{\pm}$ & full set of isomorphism class representatives of $C_2$-modules of\\
                     & the form $C(\sop^\infty)$ for $C\in \cG_{\cpt}^{\pm}$\\
  $U_R$              & $(\cO_A/\fm_{\f})^{\times}/\{\pm 1\}$, where $A$ is a product of quadratic \'etale\\
                     & $\Q_p$-algebras over $p\in \sop$\\
  $\cG_{R,\bullet}^{\pm}$ & full set of isomorphism class representatives (as defined in\\
  {\tiny $\bullet=\cpt$ or $\sop$} & Section \ref{sec:mainheuristic}) of extensions of $C$ by $U_R(\sop^\infty)$ for $C\in \cG_{\bullet}^{\pm}$
\end{tabular}

\begin{acknowledgements}
  We are very grateful to Robin Ammon for numerous corrections and comments on a preliminary draft.
  Part of this work was done while the second author was a guest at MPIM, Bonn. 
  He thanks the institute for financial support and for the inspiring atmosphere.
  Most of this research was funded by EPSRC Fellowship EP/P019188/1,
  `Cohen–Lenstra heuristics, Brauer relations, and low-dimensional manifolds'.
\end{acknowledgements}

\section{Arakelov ray class group}
\subsection{Definitions and basic properties}\label{sec:defs}
We continue using the notation of Sections \ref{sec:ray} and \ref{sec:introgoodprimes}.
We have an isomorphism
$$
(F\otimes_{\Q}\R)^\times \cong \prod_{v\in S_{\infty}}F_v^\times = \prod_{v\in S_{\infty}}(\c(F_v^\times)\times\R_{>0}),
$$
where $F_v$ denotes completion of $F$ at $v$, and $\c(F_v^\times)$ denotes the
maximal compact subgroup of $F_v^\times$, which is equal to $\{\pm 1\}$ if $v$
is real, and to the unit circle in $F_v$ if $v$ is complex. This induces the
isomorphism $\overline{F_\R^\times}\cong (\R_{>0})^{S_{\infty}}$ mentioned in
Section \ref{sec:introgoodprimes}. Let the map $\overline{F_\R^\times}\to \R_{>0}$
induced by the absolute value of the $\R$-algebra norm be denoted by $\AbsNm$.
Explicitly, it is given on ${\bfa}=(a_v)_v\in (\R_{>0})^{S_{\infty}}$ by
$\bfa \mapsto \prod_{v\in S_{\infty}}a_v^{\delta_v}$, where $\delta_v=1$ if $v$
is real, and $\delta_v=2$ if $v$ is complex. The projection
$\Id_F(\fm)\times_{\R_{>0}}\overline{F_{\R}^\times}\to \Id_F(\fm)$ induces a
surjective map $\Pic_F^0(\fm)\to \Cl_F(\fm)$. The preimage in $F^1(\fm)$ of
$1\times \ker \AbsNm\subset \Id_F(\fm)\times_{\R_{>0}}\overline{F_{\R}^\times}$
is $\cO_F^1(\fm)=\cO_F^\times\cap F^1(\fm)$, which is 
a finite index subgroup of $\cO_F^\times$, and Dirichlet's unit theorem implies 
that the quotient $(1\times \ker \AbsNm)/\im \cO_F^1(\fm)$ is a torus.
%All these observations
%also apply when $c$ is taken to be $\cO_F$.
Write subscript $\mymod \fm$ to denote the image in $(\cO_F/\fm_{\f})^\times\times\{\pm1\}$.
It follows from the observations just made that the exact sequences
$\myseq_F^{\fin}(\fm)$ and $\myseq_F^{\Ara}(\fm)$ form part of a commutative diagram of
$\Aut(F)$-modules
\begin{eqnarray*}
  \xymatrix{
    & 0\ar[d] & 0\ar[d] & 0\ar[d] &\\
  0\ar[r] & \frac{(\cO_F^\times)_{\mymod \fm}}{(\mu_F)_{\mymod \fm}}\ar[d]\ar[r] & \bT_F(\fm)\ar[d]\ar[r] & \bT_F\ar[d]\ar[r] & 0\\
  0\ar[r] & \frac{(\cO_F/\fm_{\f})^\times\times\{\pm1\}^{\fm_{\infty}}}{(\mu_F)_{\mymod \fm}}\ar[d]\ar[r] & \Pic_F^0(\fm)\ar[d]\ar[r] & \Pic_F^0\ar[d]\ar[r] & 0\\
  0\ar[r] & \frac{(\cO_F/\fm_{\f})^\times\times\{\pm1\}^{\fm_{\infty}}}{(\cO_F^\times)_{\mymod \fm}}\ar[d]\ar[r] & \Cl_F(\fm)\ar[d]\ar[r] & \Cl_F\ar[d]\ar[r] & 0\\
          & 0 & 0 & 0 &
  }
\end{eqnarray*}
with exact rows and columns, where $\bT_F(\fm) = \cO_F^1(\fm)\otimes_{\Z}\R/\Z$, and again 
$\fm$ is dropped from the notation if it is equal to $(\cO_F,\{\})$. We will denote the first 
exact row of the diagram by $\myseq_F^{\tori}(\fm)$, and the entire diagram by $\mydiag_F(\fm)$. 
In summary, $\mydiag_F(\fm)$ is a short exact sequence of exact sequences
\begin{eqnarray}\label{eq:DiagRayArak}
  \mydiag_F(\fm)\colon 0\to \myseq_F^{\tori}(\fm) \to \myseq_F^{\Ara}(\fm) \to \myseq_F^{\fin}(\fm)\to 0.
\end{eqnarray}
In Section \ref{sec:functors} we will explain how to recover the exact sequences
$\myseq^{\fin}_F(\fm)$ and $\myseq^{\tori}_F(\fm)$, and, in fact, the whole of \eqref{eq:DiagRayArak}
in a canonical way from the exact sequence $\myseq^{\Ara}_F(\fm)$ of compact groups.

\subsection{Quadratic forms and modular geodesics}\label{sec:geodesics}
This subsection is independent of the rest of the paper. We begin by briefly
recalling Schoof's definition of a measure on each connected component of
an Arakelov class group. We then recall
a well-known construction that associates with each narrow ideal class of the ring
of integers of a real quadratic number field a geodesic on the modular surface
$Y(1)=\SL_2(\Z)\backslash \bH$, where $\bH$ is the hyperbolic upper half plane, 
on which $\SL_2(\Z)$ acts by isometries. After that we come to the main content of
the subsection, a reinterpretation of this construction
in terms of Arakelov class groups: we define natural maps from narrow
Arakelov class groups of real quadratic fields to $Y(1)$ that send the connected
components of these Arakelov class groups isometrically to the corresponding
geodesics just mentioned.

It is not difficult to generalise the entire construction to more general
moduli, yielding, for every real quadratic field $F$ and every modulus
$\fm=(\fm_{\f},\fm_{\infty})$ with $\fm_{\infty}$ containing both
real places of $F$, isometries from the connected components of
$\Pic_F^0(\fm)$ to closed geodesics on $Y_1(\fm_{\f})$. Since this
discussion is somewhat tangential to the main thrust of the paper,
we leave this generalisation to the interested reader.
%This construction can be generalised in a straight forward
%manner to more general Arakelov ray class groups, but we do not spell out that
%generalisation here.
%Finally, we extend this construction to more general
%(narrow) Arakelov ray class groups of real quadratic fields.

\subsubsection*{Narrow Arakelov class group} For the rest of this subsection,
let $F$ be a real quadratic field, let $D$ be its discriminant, and let $g$
denote the non-trivial automorphism of $F$. Let $\sigma\colon F\to \R$ be one
of the two real embeddings of $F$. For an element $\alpha$ of $F$, write
$\Norm(\alpha)=\alpha\cdot g(\alpha)$, and for an ideal $\fa$ of $\cO_F$ let
$\Norm(\fa)$ denote the ideal norm of $\fa$. The narrow class group $\Cl^+_F$
of $F$ is the ray class group with respect to the modulus
$\fm=(\fm_{\f},\fm_{\infty})$ with $\fm_{\f}=\cO_F$ and with $\fm_{\infty}$
containing both real embeddings of $F$. Similarly, the narrow Arakelov class
group $\Pic^{0,+}_F$ of $F$ is the Arakelov ray class group of $F$ with respect
to the same modulus. The quotient map
$F_{\R}^\times\to \overline{F_{\R}^\times}$ induces an isomorphism from
Schoof's oriented Arakelov class group $\widetilde{\Pic}^0_F$
\cite[\S 5]{Schoof} to $\Pic^{0,+}_F$. Denote its inverse by $\raytooriented$.

\subsubsection*{Metrics on connected components of Arakelov class groups} In \cite{Schoof}
Schoof defines a metric on every connected component of the (oriented) Arakelov
class group of a number field $K$. It is induced by the restriction of the trace
form from $K_{\R}$ to $K_{\R}^0=\ker(\AbsNm\colon K_{\R}\to \R_{>0})$, using
the fact that the latter $\R$-vector space may be canonically identified with
the tangent space at the identity of $\Pic^0_K$. Here we recall Schoof's definition
in explicit terms in the special case of $\Pic^{0,+}_F$, using the isomorphism
$\raytooriented$.

Let $I\in \Id_F$, and consider two arbitrary elements $p_1$, $p_2$ on the
connected component of $\Pic_F^{0,+}$ corresponding to the class of $I$ in
$\Cl_F$ (see the commutative diagram of Section \ref{sec:defs}). These elements
may be represented by
$$
(I,(u_1,\Norm(I)u_1^{-1})), \quad (I,(u_2,\Norm(I)u_2^{-1}))\in \Id_F\times_{\R_{>0}}\overline{F_{\R}^\times},
$$
respectively, where we are using the isomorphism $\overline{F_{\R}^\times}\cong
F_{\sigma}^\times/\{\pm1\} \times F_{g\sigma}^\times/\{\pm1\}$,
and where recall that, for an Archimedean place $v\in \{\sigma,g\sigma\}$ of $F$,
$F_v$ denotes the completion of $F$ at $v$ (see Section \ref{sec:introgoodprimes}).
Moreover, if $(I,(u_1',\Norm(I)u_1'^{-1}))$
is another element of $\Id_F\times_{\R_{>0}}\overline{F_{\R}^\times}$ representing $p_1$,
then there exist $\epsilon\in \cO_F^{\times,+}=\{e\in \cO_F^\times: \sigma(e)>0, g\sigma(e)>0\}$
such that $u_1=\sigma(\epsilon)u_1'$.
In \cite[\S 6]{Schoof} Schoof defines the distance between $p_1$ and $p_2$ to be
$$
d_{\rm Sch}(p_1,p_2) = \min_{\epsilon\in \cO_F^{\times,+}}\sqrt 2\cdot\left|\log\frac{\sigma(\epsilon)u_1}{u_2}\right|.
$$
For our purposes it will be convenient to rescale Schoof's metric. We define
$$
d(p_1,p_2) = \min_{\epsilon\in \cO_F^{\times,+}}2\cdot\left|\log\frac{\sigma(\epsilon)u_1}{u_2}\right|,
$$
whenever $p_1$, $p_2$ are elements of one connected component of $\Pic^0_F$, represented
as above.

Let $\fm$ be a modulus, and let $\pi(\fm)\colon \Pic^0_F(\fm)\to \Pic^0_F$ be the
natural covering map (see Section \ref{sec:defs}). The pullback under $\pi(\fm)$
of the metrics $d$ on the connected components of $\Pic^0_F$
defines metrics on the connected components of $\Pic_F^0(\fm)$, which we will also denote
by $d$. In other words, we give the connected components of $\Pic_F^0(\fm)$
the unique metrics that make the covering map $\pi(\fm)$ a local isometry.

\subsubsection*{Ideal classes and modular geodesics}
Recall that there is a bijection between $\Cl^+_F$ and the set of
$\SL_2(\Z)$-orbits of primitive integral binary quadratic forms of discriminant
$D$, defined as follows. We call a $\Z$-basis $(\alpha,\beta)$ of a fractional
ideal of $\cO_F$ \emph{positively ordered} if one has
$\sigma(g(\alpha)\beta-\alpha g(\beta))>0$. The bijection just mentioned is
given by sending the class of a fractional ideal $\fa$, with positively
ordered $\Z$-basis $\alpha$, $\beta$, to the orbit of the binary quadratic
form $f_{\fa}(x,y)=\frac{\Norm(\alpha x-\beta y)}{\Norm\fa}$.

\begin{remark}\label{rmrk:notcanon}
  In fact, this bijection is a group isomorphism, where the group
  operation on the set of orbits of binary quadratic forms of discriminant
  $D$ is given by Dirichlet composition. This isomorphism is not quite canonical:
  if $\sigma$ is replaced by the Galois conjugate embedding, then
  this results in the bijection being composed with the
  multiplication-by-$(-1)$ automorphism.
\end{remark}

Next, recall that to each $\SL_2(\Z)$-orbit of indefinite primitive integral
binary quadratic forms one assigns a closed geodesic on $Y(1)$ as follows.
To the orbit $[f]$ of such a form $f(x,y)=ax^2+bxy+cy^2$ with
discriminant $D$
one assigns the image $\gamma_{[f]}$ under the covering map $\bH\to Y(1)$ of
the geodesic half-circle in $\bH$ with end points $\frac{-b\pm \sqrt{D}}{2a}$.
This defines a bijection between the set of $\SL_2(\Z)$-orbits of
indefinite primitive integral binary quadratic forms and the
set of closed geodesics on $Y(1)$. This construction goes back at least to Artin \cite{Artin}.

In summary, combining the above bijections, we have assigned to each class in
$\Cl_F^+$ a closed geodesic in $Y(1)$. Specifically, if a fractional ideal
has positively oriented basis $\alpha$, $\beta$, then we have assigned to its
ideal class in $\Cl_F^+$ the image in $Y(1)$ of the geodesic half-circle
in $\bH$ with end points $\sigma(\beta/\alpha)$ and $\sigma(g(\beta/\alpha))$.
We will now give an Arakelov theoretic
reinterpretation of this construction.

\subsubsection*{Arakelov class groups and $Y(1)$}
Let $\FRtoC\colon F_{\R}\to \C$ be the $\R$-vector space isomorphism given by
sending the positive (with respect to $\sigma$) square root of $D$ to the
square root of $-D$ in the upper half plane. One may associate with each
element of $\widetilde{\Pic}_F^0$ an isometry class of sublattices of $F_\R$
as follows: with a class of $\widetilde{\Pic}_F^0$ with representative
$(I,u)\in \Id_F \times_{\R_{>0}}F_{\R}^\times$ one associates the lattice
$u^{-1}I$ (see \cite[\S 5]{Schoof}, but note that our sign conventions are
different from Schoof's).
Precomposing this with the isomorphism $\raytooriented$, postcomposing with
the isomorphism $\FRtoC$, and taking the homothety class of the resulting
lattice in $\C$ defines a map $\ourisometry$ from $\Pic^{0,+}_F$ to $Y(1)$.

\begin{proposition}\label{prop:ArakelovIsometry}
For every class $c$ in $\Cl_F^+$ the map $\ourisometry$ defines an isometry
from the connected component of $\Pic^{0,+}_F$ corresponding to $c$ with the
metric $d$ defined above to the geodesic $\gamma_{[f_{\fa}]}$ in $Y(1)$ with
the hyperbolic metric, where $\fa\in \Id_F$ is a representative of $c$.
\end{proposition}
\begin{proof}
Let $\fa$ be a fractional
ideal in $F$, and let $\alpha$, $\beta$ be a positively oriented basis
of $\fa$. An arbitrary element of
the connected component of $\Pic^{0,+}_F$ corresponding to the class of $\fa$
in $\Cl_F^+$ is represented by an element of the form
$(\fa,(u^{-1},\Norm\fa\cdot u))\in \Id_F\times_{\R_{>0}}\overline{F_{\R}^\times}$, $u\in \R_{>0}$,
where we have used the isomorphism $\overline{F_{\R}^\times}\cong
F_{\sigma}^\times/\{\pm1\} \times F_{g\sigma}^\times/\{\pm1\}$.
The lattice in $F_\R=F_{\sigma}\times F_{g\sigma}$ that Schoof attaches
to such an element is
$$
\Lambda=\{(\sigma(au\alpha+bu\beta),\Norm\fa\cdot\sigma(au^{-1}g(\alpha)+bu^{-1}g(\beta)))\in F_{\R}: a,b\in \Z\}.
$$
To ease notation, given an element $\gamma$ of $F$, we write $\bar{\gamma}$
for the Galois conjugate $g(\gamma)$ of $\gamma$. Also, we will now drop
$\sigma$ from the notation, and will, for the rest of the proof, view all
elements of $F$ as elements of $\R$, using the embedding $\sigma$.
Then the image of $\Lambda$ under the $\R$-algebra isomorphism $\FRtoC\colon F_{\R}\to \C$
is
\begin{eqnarray*}
  \lefteqn{\Big\langle \frac{u\alpha+\Norm \fa\cdot u^{-1}\bar{\alpha}}{2} + \frac{u\alpha-\Norm \fa\cdot u^{-1}\bar{\alpha}}{2}i,}\\
& & \frac{u\beta+\Norm \fa\cdot u^{-1}\bar{\beta}}{2} + \frac{u\beta-\Norm \fa\cdot u^{-1}\bar{\beta}}{2}i\Big\rangle_{\Z}\subset \C.
\end{eqnarray*}
Treating, as usual,
the upper half plane $\bH$ as the moduli space of homothety classes of lattices in $\C$,
we see that the lattice $\Lambda$ corresponds to the point
$$
\frac{u\beta(1+i) + \Norm \fa u^{-1}\bar{\beta}(1-i)}{u\alpha(1+i) + \Norm \fa u^{-1}\bar{\alpha}(1-i)}\in \bH,
$$
where the fact that this is indeed a point in the upper half plane is easily seen
to follow from the hypothesis that the basis $\alpha$, $\beta$ of $\fa$ be
positively oriented.
Using the isometric action of $\PGL_2(\R)$ on $\bH$, we may rewrite this point as
$$
\begin{pmatrix}\beta&\bar{\beta}\\\alpha&\bar{\alpha}\end{pmatrix}\cdot u^2i/\Norm \fa.
$$
As $u$ runs from $0$ to $\infty$, this point traces out the geodesic half-circle
in $\bH$ with end points $\beta/\alpha$ and $\bar{\beta}/\bar{\alpha}$,
whose image in $Y(1)$ is exactly $\gamma_{[f_{\fa}]}$. Moreover,
given two such points, with $u=u_1$, respectively $u=u_2$, the hyperbolic
distance between them in $Y(1)$ is
$\min_{\epsilon\in \cO_F^{\times,+}}\left|\log\frac{(\epsilon\cdot u_1)^2}{u_2^2}\right|$,
as claimed.
\end{proof}

As mentioned at the top of the subsection, a similar construction, which we will not spell out,
yields isometries from the connected components of $\Pic^0_F(\fm)$ to closed geodesics on
the modular curve $Y_1(\fm_{\f})$, where $\fm=(\fm_{\f},\fm_{\infty})$ is any modulus with
$\fm_{\infty}$ containing both infinite places of $F$. Here, the metric on the connected components
of $\Pic^0_F(\fm)$ is the pull-back of the metric defined above under the finite covering
$\Pic^0_F(\fm)\to \Pic^{0,+}_F$.

\section{Connected component and group of components functors}\label{sec:functors}
In this section we will explain, in an abstract setting, how to recover the
sequences $\myseq_F^{\fin}(\fm)$ and $\myseq_F^{\tori}(\fm)$ and the whole diagram
\eqref{eq:DiagRayArak}, as well as the reduction map
$\cO_F^\times/\mu_F\to \frac{(\cO_F/\fm_{\f})^\times\times\{\pm1\}^{\fm_{\infty}}}{(\mu_F)_{\mymod \fm}}$ from the
exact sequence $\myseq_F^{\Ara}(\fm)$ of compact groups.
%We will then investigate some properties of certain $\Ext$-groups,
%which we will need in Section \ref{bla}.

Let $\AbGp$, $\CAbGp$, $\CCAbGp$, and $\PAbGp$ be the categories of: abelian groups
(with the discrete topology),
compact Hausdorff abelian groups, connected compact abelian groups, and totally
disconnected compact Hausdorff (i.e. profinite) abelian groups, respectively.
There are covariant functors
\begin{eqnarray*}
  \myfunc_0\colon\CAbGp & \to & \CCAbGp\\
  \myfunc^0\colon\CAbGp & \to & \PAbGp,
\end{eqnarray*}
the first of these induced by sending a compact abelian group to the connected component of
the identity,
% and a group homomorphism to the induced homomorphism on the connected components;
and the second induced by sending a compact abelian group to the
group of components.
%and a group homomorphism to the induced homomorphism on component groups.
Also, Pontryagin duality defines an involutive anti-equivalence of categories
$$
\vee\colon \CAbGp \leftrightarrow \AbGp.
$$

Next, let $\TAbGp$ and $\TFAbGp$ be full subcategories of $\AbGp$ consisting
of torsion abelian groups, respectively of torsion free abelian groups.
There are covariant functors
\begin{eqnarray*}
  \tors\colon \AbGp & \to & \TAbGp\\
  \torsfree\colon \AbGp & \to & \TFAbGp,
\end{eqnarray*}
the first taking an abelian group to its torsion subgroup,
%and a group homomorphism to the induced homomorphism between torsion subgroups;
and the second taking an abelian group to its maximal torsion-free quotient.
%, and a group homomorphism to the induced homomorphism on the torsion-free quotient.

\begin{proposition}
We have canonical isomorphisms of functors
\begin{align*}
  \tors\circ \vee = \vee\circ \myfunc^0\colon & \CAbGp\to \TAbGp,\\
  \vee\circ \tors = \myfunc^0\circ\vee\colon &\AbGp\to \PAbGp,
\end{align*}
and
$$
  \torsfree\circ \vee = \vee\circ \myfunc_0,\;\;\; \vee\circ \torsfree = \myfunc_0\circ \vee.
$$
\end{proposition}
\begin{proof}
Firstly, a compact Hausdorff abelian group is connected if and
only if its dual is torsion free \cite[Cor. 4 of Thm. 30]{PontrDual}.
%
%First we claim that $G$ is connected if
%and only if $G^{\vee}$ is torsion free. Suppose that $G$ is connected, and
%let $\chi \in G^{\vee}$ have finite order. Then in particular $\im(\chi)$ is
%finite. Therefore $\chi^{-1}(0)$ is both open and closed in $G$, hence
%$\chi^{-1}(0)=G$, i.e. $\chi=0$. Conversely suppose that $G^{\vee}$ is torsion
%free, equivalently that the natural map
%$G^{\vee} \to G^{\vee} \otimes_{\Z} \Q$ is injective. Dualising 
%we obtain a continuous epimorphism
%$$
%(\bA_{\Q}/\Q)^{{\dim}_{\Q}(G^{\vee} \otimes_{\Z} \Q)} \twoheadrightarrow G,
%$$
%where $\bA_{\Q}$ is the ad\`ele group of $\Q$. The group $\bA_{\Q}/\Q$ is
%connected, and powers and continuous images of connected spaces are connected,
%so $G$ is connected.
%
Second, a compact Hausdorff abelian group 
is totally disconnected if and only if the dual is a torsion group
\cite[Cor. 1 of Thm. 30]{PontrDual}.
The statement easily follows from these two observations.
\end{proof}

\begin{remark}\label{rmk:nonexact}
None of the functors $\myfunc_0$, $\myfunc^0$, $\tors$, $\torsfree$ are exact. Indeed,
$\myfunc^0$ is right exact, but not left exact; while $\myfunc_0$ preserves
epimorphisms and monomorphisms, but does not preserve exactness:
for example the connected component of the kernel of an isogeny between two
tori is just the trivial group. Dually, $\tors$ is left exact, but not
right exact; while $\torsfree$ preserves monomorphisms and epimorphisms,
but does not preserve exactness: consider, for example, the effect
of $\torsfree$ on a map from a free abelian group to a finite group.
\end{remark}
It is thanks to Remark \ref{rmk:nonexact} and its consequences that the exact
sequence (\ref{eq:SeqRayArak}) ``knows'' about the image of the map
$\cO_F^\times/\mu_F\to \frac{(\cO_F/\fm_{\f})^\times\times\{\pm1\}^{\fm_{\infty}}}{(\mu_F)_{\mymod \fm}}$.
We will now make this precise.

Given an exact sequence
$$
\Theta\colon 0\to A\to B\to C\to 0
$$
in $\CAbGp$, define
$$
\myfunc_0(A,B) = \ker(\myfunc_0(B)\to \myfunc_0(C))
$$
and
$$
\myfunc^0(A,B) = \ker(\myfunc^0(B)\to \myfunc^0(C)).
$$
We then have natural exact sequences
\begin{align*}
  \myfunc_0^0(\Theta)\colon & 0\to \myfunc_0(A,B)\to A \to \myfunc^0(A,B)\to 0,\\
  \myfunc_0(\Theta) \colon & 0\to \myfunc_0(A,B)\to \myfunc_0(B)\to \myfunc_0(C)\to 0,\\
  \myfunc^0(\Theta) \colon & 0\to \myfunc^0(A,B)\to \myfunc^0(B)\to \myfunc^0(C)\to 0,
\end{align*}
and an exact sequence of exact sequences
$$
\mydiag(\Theta): 0\to \myfunc_0(\Theta)\to \Theta \to \myfunc^0(\Theta) \to 0.
$$
\begin{proposition}
  Let $F$ be a number field, and $\fm$ a modulus of $F$. Then we have
  $$
    \mydiag(\myseq_F^{\Ara}(\fm)) = \mydiag_F(\fm).
  $$
\end{proposition}
\begin{proof}
  The torus $\bT_F(\fm)$ is connected, and the quotient $\Pic_F^0(\fm)/\bT_F(\fm)=\Cl_F(\fm)$
  is finite, so we have $\myfunc_0(\Pic_F^0(\fm))=\bT_F(\fm)$ and $\myfunc^0(\Pic_F^0(\fm))=\Cl_F(\fm)$.
  The result follows easily.
\end{proof}

Dually, given an exact sequence
$$
\Phi\colon 0\to X\to Y\to Z\to 0
$$
in $\AbGp$, define 
$$
\tors(Y,Z) = \coker(\tors(X)\to \tors(Y)),
$$
and 
$$
\torsfree(Y,Z) = \coker(\torsfree(X)\to \torsfree(Y)).
$$
We then have natural exact sequences
\begin{align*}
  \tors-\torsfree(\Phi)\colon & 0\to \tors(Y,Z)\to Z\to \torsfree(Y,Z)\to 0,\\
  \tors(\Phi)\colon & 0\to \tors(X)\to \tors(Y)\to \tors(Y,Z)\to 0,\\
  \torsfree(\Phi) \colon & 0\to \torsfree(X)\to \torsfree(Y)\to \torsfree(Y,Z)\to 0,
\end{align*}
and an exact sequence of exact sequences
$$
\mydiag^\vee(\Phi)\colon 0\to \tors(\Phi)\to \Phi\to \torsfree(\Phi)\to 0.
$$
One immediately verifies that if $\Theta$ is an exact sequence in $\CAbGp$,
then one has
$$
(\mydiag(\Theta))^\vee = \mydiag^\vee(\Theta^\vee).
$$

We now explain how to explicitly recover the reduction map
$\cO_F^\times/\mu_F\to \frac{(\cO_F/\fm_{\f})^\times\times\{\pm1\}^{\fm_{\infty}}}{(\mu_F)_{\mymod \fm}}$
from the exact sequence of compact groups $\myseq_F^{\Ara}(\fm)$.

Let $A$ be a finite abelian group and $C$ a compact abelian real Lie group.
One may define the group $\Ext^1(C,A)$, parametrising short exact sequences
$$
0\to A\to B\to C\to 0
$$
in $\CAbGp$ (see \cite{Moskowitz}, where the relevant homological algebra
is developed in much greater generality). It is canonically isomorphic to
$\Ext^1(A^\vee,C^\vee)$, and for our purposes one may take the latter to be its definition.

Let $\myfunc_0(C) = V/\Lambda$, where $V$ is a finite-dimensional real vector space,
and $\Lambda\subset V$ is a full rank sublattice. Let $\Lambda^*=\Hom(\Lambda,\Z)$.
There is a canonical isomorphism $\torsfree(C^\vee) = \Lambda^*$. By the universal
coefficient theorem \cite[Ch. 5, \S 6, Exercise 5]{Brown}, there is a canonical short exact sequence
$$
0\to \Ext^1(A^\vee,\Lambda^*) \to H^2(A^\vee,\Lambda^*) \to \Hom(\wedge^2 A^\vee, \Lambda^*)\to 0.
$$
More explicitly, the first map is given by sending an extension to the corresponding commutator
pairing, whence the exactness in the first and in the second term is obvious.
Since $\wedge^2A^\vee$ is finite and $\Lambda^*$ is torsion-free, the last term in that
sequence is $0$, and the first map is therefore an isomorphism. Next, from the long
exact cohomology sequence associated with the short exact sequence
$$
0\to \Lambda^* \to \Lambda^*\otimes_{\Z}\R\to \Lambda^*\otimes_{\Z}(\R/\Z) \to0
$$
one obtains canonical isomorphisms
$$
H^2(A^\vee,\Lambda^*) = H^1(A^\vee, \Lambda^*\otimes_{\Z}(\R/\Z))=\Hom(A^\vee, \Lambda^*\otimes_{\Z}(\R/\Z))=\Hom(\Lambda,A),
$$
where the first equality follows from the fact that $\Lambda^*\otimes_{\Z}\R$
is an injective $\Z$-module, and the last equality is given by Pontryagin duality.
Composing the natural map $\Ext^1(C,A)\to \Ext^1(\myfunc_0(C),A)=\Ext^1(A^\vee,\Lambda^*)$
with the above isomorphisms gives a natural map
$$
\varphi\colon \Ext^1(C,A)\to \Hom(\Lambda,A).
$$
We can give a more direct description of $\varphi$ as follows. Let
$$
  0\to A\to B\to C\to 0
$$
be a short exact sequence in $\CAbGp$ representing a class
$\Theta\in \Ext^1(C,A)$, and let $V_B$ be a universal cover of $\myfunc_0(B)$. The map
$\myfunc_0(B) \twoheadrightarrow \myfunc_0(C)$ lifts to an isomorphism
$f\colon V_B\to V$. The inverse map $f^{-1}$, restricted to $\Lambda\subset V$,
defines an element of $\Hom(\Lambda,A)$, and $\varphi(\Theta)$
is that element. The following is now immediate.
\begin{proposition}\label{prop:reductions}
  Let $F$ be a number field, and $\fm$ a modulus of $F$. Then the map
  $\varphi(\myseq_F^{\Ara}(\fm))$ is the natural reduction map
  $\cO^\times_F/\mu_F \to \frac{(\cO_F/\fm_{\f})^\times\times\{\pm1\}^{\fm_{\infty}}}{(\mu_F)_{\mymod \fm}}$.
\end{proposition}

Suppose now that, as above, $A$ is a finite abelian group, and $C$
is a compact abelian real Lie group, with $\myfunc_0(C)=V/\Lambda$,
where $V$ is a finite dimensional real vector space and $\Lambda\subset V$
is a full rank sublattice, and suppose in addition that $A$, $V$, and $C$ are equipped
with an action of $C_2$ by continuous group automorphisms such that $C_2\Lambda = \Lambda$
and such that the restriction of that action from $C$ to $\myfunc_0(C)$ is
induced by the action on $V$.
%To ease notation, suppose that $A=A_{(S)}$ and $C=_{(S)}$.
Then the above discussion applies
without change to these objects as $C_2$-modules rather than merely as groups,
and yields a map 
\begin{eqnarray}\label{eq:varphi}
  \quad\quad\varphi\colon \Ext^1_{C_2}(C,A)\to \Hom_{C_2}(\Lambda,A),
\end{eqnarray}
where $\Hom_{C_2}$ and $\Ext^1_{C_2}$ denotes homomorphisms, respectively $\Ext^1$, in the
category of topological $C_2$-modules.
Upon passing to $\sop^\infty$-torsion, we also get the analogous map
\begin{eqnarray}\label{eq:varphiS}
  \quad\quad\varphi_{\sop}\colon \Ext^1_{\Z_{\sop}[C_2]}(C(\sop^\infty),A(\sop^\infty))\to \Hom_{\Z_{\sop}[C_2]}(\Lambda_{\sop},A(\sop^\infty)).
\end{eqnarray}
We have the following result.

\begin{proposition}\label{prop:surjhom}
  The map $\eqref{eq:varphiS}$ is a surjective group homomorphism, whose kernel is the image
  of the map 
  $$
    \pi_C^*\colon \Ext_{\Z_{\sop}[C_2]}^1(\myfunc^0(C)(\sop^\infty),A(S^\infty))\to \Ext^1_{\Z_{\sop}[C_2]}(C(S^\infty),A(\sop^\infty))
  $$
  induced by the quotient map $\pi_C\colon C(\sop^\infty)\to \myfunc^0(C)(\sop^\infty)$.
\end{proposition}
\begin{proof}
  The map $\varphi_{\sop}$ is a composition of the two maps
  \begin{eqnarray*}
    \Ext^1_{\Z_{\sop}[C_2]}(C(\sop^\infty),A(\sop^\infty)) & \to & \Ext^1_{\Z_{\sop}[C_2]}(\myfunc_0(C)(\sop^\infty),A(\sop^\infty))\\
                                                           & \to & \Hom_{\Z_{\sop}[C_2]}(\Lambda_{\sop},A(\sop^\infty)).
  \end{eqnarray*}
  It follows from the discussion preceding Proposition \ref{prop:reductions} that
  the second of these is an isomorphism, and it suffices to argue that the first is surjective
  with the claimed kernel.

  The claim regarding the kernel is clear, and we now prove surjectivity.
  Let $g$ be the non-trivial element of $C_2$.
  We have ring isomorphisms $\Z_{\sop}[C_2]\to \Z_{\sop}[C_2]/(g-1)\times \Z_{\sop}[C_2]/(g+1)\cong \Z_{\sop}\times \Z_{\sop}$,
  and correspondingly every $\Z_{\sop}[C_2]$-module decomposes as a direct sum of two $\Z_{\sop}$-modules.
  Thus we may ignore the $C_2$-module structure and work in the category of abelian groups,
  so that surjectivity of the first map follows from the fact that $\Ext^2$ of finite groups vanishes
  \cite[Lemma 3.3.1]{Weibel}.
\end{proof}

We close the section by describing a set that conveniently parametrises the pairs
$(\myfunc_0(\Theta)(\sop^\infty),\myfunc^0(\Theta)_{\sop})$ for $\Theta\in \Ext^1_{C_2}(C,A)$,
and by proving a result that will be important in the proof of Theorem \ref{thm:realquad}.

Let $A$, $C$, $\Lambda$, $V$ be as above. Let $\pairs_{\sop}(C,A)$ be the set of pairs
$(\psi, \phi)$, where $\psi\in \Hom_{\Z_{\sop}[C_2]}(\Lambda_{\sop},A(\sop^\infty))$
and $\phi \in \Ext^1_{\Z_{\sop}[C_2]}(\myfunc^0(C)(\sop^\infty),A(\sop^\infty)/\Im\psi)$.
There is a natural map
$$
\pairsmap\colon \Ext^1_{\Z_{\sop}[C_2]}(C(\sop^\infty),A(\sop^\infty))\to \pairs_{\sop}(C,A)
$$
given by
sending $\Theta\in \Ext^1_{\Z_{\sop}[C_2]}(C(\sop^\infty),A(\sop^\infty))$ to the
pair consisting of $\psi=\varphi_{\sop}(\Theta)$ and of the class
$\phi=\myfunc^0(\Theta)(\sop^\infty)\in\Ext^1_{\Z_{\sop}[C_2]}(\myfunc^0(C)(\sop^\infty),A(\sop^\infty)/\Im\psi)$.

\begin{proposition}\label{prop:extpreimage}
The map $\pairsmap$ is surjective. Moreover, the preimages in $\Ext^1_{C_2}(C(\sop^\infty),A(\sop^\infty))$
of any two elements of $\pairs_{\sop}(C,A)$ of the form $(\psi,\phi)$ and $(\psi,\phi')$
have the same size.
\end{proposition}
\begin{proof}
  %By Proposition \ref{prop:surjhom}, the composition of the map $\pairsmap$
  %with the projection onto $\Hom_{\Z_{\sop}[C_2]}(\Lambda_{\sop},A(\sop^\infty))$
  %is surjective. To prove the remaining assertions, we will explicitly describe
  %the preimage of an arbitrary element of $\pairs_{\sop}(C,A)$.
%
  Let $(\psi,\phi)\in \pairs_{\sop}(C,A)$ be arbitrary, and let $\pi_A\colon A(\sop^\infty)\to A(\sop^\infty)/\Im\psi$
  and $\pi_C\colon C(\sop^\infty)\to \myfunc^0(C)(\sop^\infty)$ be the respective quotient maps. By the usual
  functoriality of $\Ext$, they induce maps in the commutative diagram
  \[
    \xymatrix{
      \Ext^1_{\Z_{\sop}[C_2]}(\myfunc^0(C)(\sop^\infty),A(S^\infty)) \ar[r]^{\pi_{C,1}^*}\ar[d]_{(\pi_{A})_*^1} & \Ext^1_{\Z_{\sop}[C_2]}(C(S^\infty),A(S^\infty))\ar[d]^{(\pi_{A})_*^2}\\
      \Ext^1_{\Z_{\sop}[C_2]}(\myfunc^0(C)(\sop^\infty),A(\sop^\infty)/\Im\psi)\ar[r]^{\pi_{C,2}^*} &
    \Ext^1_{\Z_{\sop}[C_2]}(C(\sop^\infty),A(\sop^\infty)/\Im\psi).
    }
  \]
%  \begin{align*}
%    (\pi_A)_*\colon & \Ext^1_{C_2}(C(\sop^\infty),A(\sop^\infty)) \to \Ext^1_{C_2}(C(\sop^\infty),A(\sop^\infty)/\Im\psi),\\
%    (\pi_C)^*\colon & \Ext^1_{C_2}(\myfunc^0(C)_S,A(\sop^\infty)/\Im\psi)\to
%    \Ext^1_{C_2}(C(\sop^\infty),A(\sop^\infty)/\Im\psi).
%  \end{align*}

  We claim that $(\pi_A)_*^1$ and $(\pi_A)_*^2$ are both surjective.
  Indeed, the cokernel of $(\pi_A)_*^1$ canonically embeds
  in $\Ext^2_{\Z_{\sop}[C_2]}((\Im\psi)^\vee,C^\vee(\sop^\infty))$,
  which vanishes by the same argument as in the proof of Proposition \ref{prop:surjhom},
  and the argument for $(\pi_A)_*^2$ is similar.

  It follows from the standard formalism of the $\Ext$ functor that the classes
  in $\Ext^1_{\Z_{\sop}[C_2]}(C(\sop^\infty),A(\sop^\infty))$ to which $\phi$ can be lifted are exactly
  those that are mapped, under $(\pi_A)_*^2$, to $(\pi_{C,2})^*(\phi)$. The %remaining
  assertions therefore follow from the fact that $(\pi_A)_*^1$ and $(\pi_A)_*^2$ are surjective
  group homomorphisms and that, by Proposition \ref{prop:surjhom},
  the image of $\pi_{C,1}^*$ is exactly the kernel of $\phi_S$.
\end{proof}
%\begin{remark}\label{rmrk:S-local}
%  The construction of $\sigma$ and the statement of Proposition \ref{prop:extpreimage}
%  are, in a sense, purely ``$S$-semi-local''. More precisely, they do not depend on
%  the fact that $C(S^{\infty})$ is the $S^\infty$-torsion subgroup of a compact
%  group: we may replace $V$ by a free $(\prod_{p\in S}\Q_p)$-module $V_S$ with a
%  $C_2$-action, $\Lambda$ by a $C_2$-stable free full rank $(\prod_{p\in S}\Z_p)$-submodule $\Lambda_S$,
%  and $C$ by an extension of a finite $S$-group by $V_S/\Lambda_S$, and all
%  the ``$S$-semi-local'' parts of the above discussion still apply.
%  In the next section we will apply this discussion to $S^\infty$-torsion subgroups
%  of Arakelov class groups. While these do happen to look like the groups $C$ in
%  this section, we will model them as ``random'' $S^\infty$-torsion groups
%  of a certain kind, and for those there is no reason to treat them as
%  $S^\infty$-torsion subgroups of compact groups.
%\end{remark}

\section{Arakelov ray class groups at the ``good'' primes}\label{sec:good}
In this section we prove the existence of the probability distribution used in
Heuristic \ref{he:realquad} and derive several concrete consequences of the heuristic,
in particular proving a precise version of Theorem \ref{thm:realquad}.

Let $\sop$, $\cG_{\cpt}^{\pm}$, $\cG_{\sop}^{\pm}$, $\fm_{\f}$, $A$, $R$, $U_R$, $\cG_{R,\sop}^{\pm}$,
$\cG_{R,\cpt}^{\pm}$, and $\bP^{\pm}$
be as in Section \ref{sec:mainheuristic}. Recall that
$\cG_{\cpt}^{\pm}$ is a full set of isomorphism class representatives of finite abelian groups
(when the sign is $-$), respectively of extensions of $\R/\Z$ be a finite abelian group (when the sign is $+$),
in both cases with $C_2$ acting by multiplication by $-1$, and $\cG_{\sop}^{\pm}$ is a full set
of isomorphism class representatives of $\Z_{(\sop)}[C_2]$-modules of
the form $C(\sop^\infty)$ for $C\in \cG_{\cpt}^{\pm}$.
%finite abelian $\sop$-groups (when the
%sign is $-$), respectively of extensions of $\prod_{p\in \sop}\Q_p/\Z_p$ by
%finite abelian $\sop$-groups (when the sign is $+$), in both cases with $C_2$
%acting by multiplication by $-1$,
Recall also that $\cG_{R,\cpt}^{\pm}$ and $\cG_{R,\sop}^{\pm}$
are full sets of isomorphism classes of extensions $\Theta$ of elements of $\cG_{\cpt}^{\pm}$,
respectively of $\cG_{\sop}^{\pm}$, by $U_R(\sop^\infty)$. Recall that
when defining the notion of isomorphism of such extensions, we only consider
$C_2$-module automorphisms of $U_R(\sop^\infty)$ that arise as restrictions of ring
automorphisms of $R$. Let $\Aut_{\ring}U_R(\sop^\infty)$ denote the group of all such
automorphisms $U_R(\sop^\infty)$. For $C\in \cG_{\sop}^{\pm}$, define, by analogy with the previous
section, $\myfunc_0(C)$ to be the maximal divisible subgroup
of $C$ (which is trivial if $C\in \cG^-$, and is isomorphic to $\prod_{p\in \sop}\Q_p/\Z_p$
if $C\in \cG^+$), and define $\myfunc^0(C)=C/\myfunc_0(C)$.

\subsection{Equidistribution predictions}\label{sec:equidistr}
%Let $\sop$, $\cG^{\pm}$, $\fm_{\f}$, $A$, $R$, $\cG_R^{\pm}$, $U_R$, and $\bP^{\pm}$
%be as in Section \ref{sec:mainheuristic}. Recall that $\cG^{\pm}$ is a full set
%of isomorphism class representatives of finite abelian $\sop$-groups (when the
%sign is $-$), respectively of extensions of $\prod_{p\in \sop}\Q_p/\Z_p$ by
%finite abelian $\sop$-groups (when the sign is $+$), in both cases with $C_2$
%acting by multiplcation by $-1$, and $\cG_R^{\pm}$ is a full set of isomorphism
%classes of extensions $\Theta$ of elements of $\cG^{\pm}$ by $U_R$. Recall that
%when defining the notion of isomorphism of such extensions, we only consider
%$C_2$-module automorphisms of $U_R$ that arise as restrictions of ring
%automorphisms of $R$. Let $\Aut_{\ring}(U_R)$ denote the group of all such
%automorphisms $U_R$.
For $C\in \cG_{\cpt}^{\pm}$ and for an extension $\Theta\in \cG_{R,\cpt}^\pm$ of $C$ by $U_R(\sop^\infty)$,
let $O(\Theta)$ be the orbit in $\Ext^1_{C_2}(C,U_R(\sop^\infty))$ of the class
of $\Theta$ under the natural action of $\Aut C \times \Aut_{\ring}U_R(\sop^\infty)$,
and similarly for $C\in \cG_{\sop}^{\pm}$.
Note that two extensions are in the same orbit under this action if and only
if they are isomorphic.
%TODO: have to evaluate \#Aut C before tensoring with Z_S, but state the \#Aut Theta result
%on the level of the push-forward, i.e. basically as a commensurability result

\begin{lemma}\label{lem:sizeautC}
Let $C\in\cG_{\cpt}^+$. Then one has
  \begin{eqnarray*} 
    \#\Aut C & = & 2\cdot \#\Aut \myfunc^0(C)\cdot \#\myfunc^0(C).
  \end{eqnarray*}
\end{lemma}
\begin{proof}
  Since $\myfunc_0(C)$ is characteristic in $C$, we have a natural map
  $\Aut C \to \Aut \myfunc_0(C) \times \Aut \myfunc^0(C)$. Moreover,
  since $\myfunc_0(C)$ is an injective group (and $C_2$-module), the exact sequence
  \begin{eqnarray}\label{eq:extC}
    0\to \myfunc_0(C)\to C\to \myfunc^0(C)\to 0,
  \end{eqnarray}
  splits, so the map of automorphism groups just mentioned is surjective.
  Its kernel is canonically identified with $\Hom(\myfunc^0(C),\myfunc_0(C))$.
  Here the inclusion $\Hom(\myfunc^0(C),\myfunc_0(C))\hookrightarrow \Aut C$
  is given by $f\mapsto (x\mapsto x+f(\bar{x}))$, where, for $x\in C$, we denote
  by $\bar{x}$ the image of $x$ in $\myfunc^0(C)$.
  Now, $\myfunc_0(C)$ is isomorphic to $\R/\Z$, so we have $\#\Hom(\myfunc^0(C),\myfunc_0(C)) = \#\myfunc^0(C)$
  (recall that all our automorphisms and homomorphisms are $C_2$-equivariant, but
  that does not change the above conclusion, since the non-trivial element of $C_2$
  acts by $-1$ on $C$), and further we have $\#\Aut\myfunc_0(C) = 2$, which
  proves the lemma.
\end{proof}

\begin{lemma}\label{lem:HomExt}
  Let $C\in \cG_{\cpt}^{\pm}$. Then we have
  \begin{eqnarray}\label{eq:HomExt}
    \#\Hom_{C_2}(U_R(\sop^\infty)^\vee,C^\vee)=\#\Ext^1_{C_2}(U_R(\sop^\infty)^\vee,C^\vee).
  \end{eqnarray}
\end{lemma}
\begin{proof}
  First, note that since $\sop$ contains only odd primes, one has
  $\Z_{\sop}[C_2] = \Z_{\sop}[C_2]/(1-g) \times \Z_{\sop}[C_2]/(1+g) \cong \Z_{\sop}\times \Z_{\sop}$,
  where $g$ denotes the generator of $C_2$. Temporarily, write 
  \begin{eqnarray*}
    U_+ &=& \Z_{\sop}[C_2]/(1-g)\otimes_{\Z_{\sop}[C_2]}U_R(\sop^\infty)^\vee,\\
    U_- &=& \Z_{\sop}[C_2]/(1+g)\otimes_{\Z_{\sop}[C_2]}U_R(\sop^\infty)^\vee,
  \end{eqnarray*}
  so that we have a decomposition of $C_2$-modules
  $U_R(\sop^\infty)^\vee = U_+ \oplus U_-$, and accordingly a decomposition
  of the $\Hom$ and $\Ext$ groups in \eqref{eq:HomExt}. We treat the plus and the minus parts
  separately.

  Since $1-g$ annihilates $U_+$ and acts by multiplication by $2$ on $C^\vee$, the image of
  every element of $\Hom_{C_2}(U_+,C^\vee)$ lands in the $2$-torsion of $C^\vee$. But
  the order of $U_+$ is co-prime to $2$, so in fact we have $\Hom_{C_2}(U_+,C^\vee)=0$. Correspondingly,
  we argue that we also have $\Ext^1_{C_2}(U_+,C^\vee)=0$.
  %We have a direct sum of groups
  %$C^\vee = C^\vee(2^\infty)\oplus C^\vee(2')$, where $C^\vee(2^\infty)$ is the $2$-Sylow
  %subgroup of the torsion subgroup of $C^\vee$, and $C^\vee(2')$ is a complement. Since
  %$g$ acts by $-1$ on all of $C^\vee$, that decomposition is also a decomposition of $C_2$-modules.
  %Now, we clearly have $\Ext^1{C_2}(U_+,C^\vee(2^\infty))=0$, since the two modules have co-prime
  %orders. It therefore suffices to show that we also have $\Ext^1{C_2}(U_+,C^\vee(2'))=0$.
  To that end, let $B$ be any extension of $U_+$ by $C^\vee$. Since $g$ acts by multiplication
  by $-1$ on $C^\vee$, the endomorphism $g+1$ of $B$ annihilates $C^\vee$, and hence defines
  a map $U_+\to B$. Since $U_+$ has odd order, multiplication by $2$ is an automorphism of the image
  of $1+g$ in $B$, so that we have a well-defined map $\tfrac12(1+g)\colon U_+\to B$.
  Since $g$ acts trivially on $U_+$, this map defines a splitting for the extension.
  Since $B$ was arbitrary, this shows that we have $\Ext^1_{C_2}(U_+,C^\vee)=0$.
  % , and consider the action of $1+g$ on $B$.

  Finally, we now show that we have $\#\Hom_{C_2}(U_-,C^\vee)=\#\Ext^1_{C_2}(U_-,C^\vee)$.
  Since $g$ acts by $-1$ on $C^\vee$, we have $\Hom_{C_2}(U_-,C^\vee)=\Hom(U_-,C^\vee)$.
  %In the category of finitely generated abelian groups we have $\#\Hom=\#\Ext^1$,
  %so it suffices to show that
  We claim that we also have $\Ext^1_{C_2}(U_-,C^\vee)=\Ext^1(U_-,C^\vee)$.
  There is an obvious injective map $\Ext^1(U_-,C^\vee)\to \Ext^1_{C_2}(U_-,C^\vee)$,
  taking an extension of abelian group to the extension with the same underlying
  group and with $C_2$ acting by $-1$. This map is also surjective: indeed,
  let $B$ be a $C_2$-module extension of $U_-$ by $C^\vee$, and consider the endomorphism $1+g$ of $B$.
  It annihilates $C^\vee$, so defines a map $U_-\to B$. But also, since $g$ also acts by $-1$
  on $U_-$, the endomorphism $1+g$ composed with the projection to $U_-$ is $0$,
  i.e. the image of $1+g$ lands in $C^\vee$. But then $(1+g)^2=2+2g$ annihilates $B$,
  so that the image of $1+g$ lands in $B[2]$. However, $U_-$ has odd order, so we conclude
  that $1+g$ is the $0$ map, so that $B$ is a module over $\Z[C_2]/(1+g)\cong\Z$.
  Thus, we have $\Ext^1_{C_2}(U_-,C^\vee)=\Ext^1(U_-,C^\vee)$, as claimed,
  and hence
  \begin{align*}
    \#\Hom_{C_2}(U_-,C^\vee)&=\#\Hom(U_-,C^\vee)\\
                            &=\#\Ext^1(U_-,C^\vee)\\
                            &= \#\Ext^1_{C_2}(U_-,C^\vee).
  \end{align*}
\end{proof}
\begin{proposition}\label{prop:sizeaut}
  Let $C\in \cG_{\cpt}^{\pm}$, and let $\Theta\in \cG_{R,\cpt}^\pm$ be an extension of $C$
  by $U_R(\sop^\infty)$. Then one has
  $$
    \#\Aut_{\ring} \Theta = \#\Aut_{\ring}U_R(\sop^\infty)\cdot \#\Aut C \cdot \frac{\#\Ext^1_{C_2}(C,U_R(\sop^\infty))}{\#O(\Theta)}.
  $$
%  Moreover, for $C\in\cG_{\cpt}^+$ one has
%  \begin{eqnarray*} 
%    \#\Aut C & = & 2\cdot \#\Aut \myfunc^0(C)\cdot \#\myfunc^0(C).
%  \end{eqnarray*}
\end{proposition}
\begin{proof}
  There is a natural group homomorphism $\Aut_{\ring}\Theta\to \Aut_{\ring}U_R(\sop^\infty)\times\Aut C$,
  given by restricting an automorphism to $U_R(\sop^\infty)$ and projecting modulo $U_R(\sop^\infty)$.
  Its image is precisely the stabiliser $\Stab \Theta$ of the class of
  $\Theta$ in $\Ext^1_{C_2}(C,U_R(\sop^\infty))$ under the natural $(\Aut_{\ring}U_R(\sop^\infty)\times\Aut C)$-action,
  while its kernel is canonically identified with $\Hom_{C_2}(C,U_R(\sop^\infty))$. Here
  the canonical inclusion $\Hom_{C_2}(C,U_R(\sop^\infty))\hookrightarrow \Aut_{\ring} \Theta$ is given
  by $f\mapsto (x\mapsto x+f(\bar{x}))$, where, for $x\in \Theta$, we denote by $\bar{x}$ the
  image of $x$ in $C$. Thus,
  we have
  \begin{eqnarray*}
    \#\Aut_{\ring}\Theta & = & \#\Stab \Theta \cdot \#\Hom_{C_2}(C,U_R(\sop^\infty))\\
                 & = &
    \#\Aut_{\ring}U_R(\sop^\infty)\cdot \#\Aut C \cdot\frac{\#\Hom_{C_2}(C,U_R(\sop^\infty))}{\#O(\Theta)}.
  \end{eqnarray*}
  The result follows by Lemma \ref{lem:HomExt}.
\end{proof}

Define
$$
c^-=\prod_{p\in \sop}\prod_{k=1}^\infty(1-p^{-k})^{-1},\quad\quad c^+=\tfrac12\prod_{p\in \sop}\prod_{k=2}^\infty(1-p^{-k})^{-1}.
$$

\begin{corollary}\label{cor:converge}
The sums
$$
\sum_{\Theta\in \cG_{R,\cpt}^\pm}\frac{1}{\#\Aut_{\ring} \Theta}
$$
converge to $c^{\pm}_R=\frac{c^{\pm}}{\#\Aut_{\ring} U_R(\sop^\infty)}$.
\end{corollary}
\begin{proof}
  Let $C\in \cG_{\cpt}^{\pm}$. As remarked at the beginning of Section \ref{sec:equidistr},
  a full set of isomorphism class representatives of extensions
  $0\to U_R(\sop^\infty)\to B\to C\to 0$ is in bijection
  with the set of $(\Aut_{\ring}U_R(\sop^\infty)\times \Aut C)$-orbits in
  $\Ext^1_{C_2}(C,U_R(\sop^\infty))$. We
  therefore deduce from Proposition \ref{prop:sizeaut} and Lemma \ref{lem:sizeautC} the equalities
  \begin{eqnarray*}
    \lefteqn{\sum_{\Theta\in \cG^{\pm}_{R,\cpt}}\frac{1}{\#\Aut_{\ring} \Theta} =}\\
    & & = \sum_{C\in \cG_{\cpt}^{\pm}}\sum_{\Theta}\frac{1}{\#\Aut_{\ring}U_R(\sop^\infty)\cdot\#\Aut C}\cdot\frac{\#O(\Theta)}{\#\Ext^1_{C_2}(C,U_R(\sop^\infty))}\\
        & & = \frac{1}{\#\Aut_{\ring} U_R(\sop^\infty)}\sum_{C\in \cG_{\cpt}^{\pm}}\frac{1}{\#\Aut C}\\
        & & = \frac{1}{\#\Aut_{\ring} U_R(\sop^\infty)}\cdot
        \bigleftchoice{\sum_{C\in \cG_{\cpt}^{+}}1/(2\cdot\#\Aut\myfunc^0(C)\cdot\#\myfunc^0(C))}{\text{ sign }+}
        {\sum_{C\in \cG_{\cpt}^{-}}(1/\#\Aut C)}{\text{ sign }-}\\
        & & = \frac{c^{\pm}}{\#\Aut_{\ring}U_R(\sop^\infty)},
  \end{eqnarray*}
  where in the first equality, the inner sum runs over representatives of
  the distinct $(\Aut_{\ring}U_R(\sop^\infty)\times \Aut C)$-orbits on $\Ext^1_{C_2}(C,U_R(\sop^\infty))$,
  and where the last equality follows from \cite[Corollary 3.7]{CL} (for sign
  $-$ this identity is in fact due to Hall \cite{Hall}).
\end{proof}
Recall from Section \ref{sec:mainheuristic} that we define a discrete
probability distribution $\bP_{\cpt}^\pm$ on $\cG_{R,\cpt}^{\pm}$ by
$\bP_{\cpt}^{\pm}(\{\Theta\}) = \frac{(c_R^{\pm})^{-1}}{\#\Aut_{\ring}\Theta}$,
and let $\bP^{\pm}$ be the pushforward of $\bP_{\cpt}^{\pm}$ under the
function $\cG_{R,\cpt}^{\pm}\to \cG_{R,\sop}^{\pm}$ induced
by sending a sequence to its $\sop^\infty$-torsion.

The next five corollaries are immediate consequences of Propositions
\ref{prop:sizeaut}, \ref{prop:extpreimage}, and \ref{prop:reductions}, and taken
together amount to a precise version of Theorem \ref{thm:realquad}.

Let $\cE$ be a full set of representatives of finite
abelian $\sop$-groups $E$. Consider the two discrete probability distributions 
$\mu^{\pm}_{\CL}$ on $\cE$ given, for all $E\in \cE$, by
\begin{eqnarray*}
  \mu^-_{\CL}(\{E\}) & = & \frac{(c^-)^{-1}}{\#\Aut E}\\
  \mu^+_{\CL}(\{E\}) & = & \frac{(2c^+)^{-1}}{\#\Aut E\cdot\#E}.
\end{eqnarray*}

These are the distributions that Cohen and Lenstra conjectured to model
the behaviour of $\sop$-parts of class groups of imaginary (for sign $-$),
respectively real (for sign $+$) quadratic number fields. The next result,
which is an immediate consequence of Proposition \ref{prop:sizeaut},
shows that Heuristic \ref{he:realquad} implies the Cohen--Lenstra heuristic
for class groups of quadratic fields.
\begin{corollary}\label{cor:PFclassgroup}
  For each of the signs $\pm$, let
  $f^{\pm}\colon \cG_{R,\sop}^{\pm}\to \cE$ be the function that sends an
  exact sequence $\Theta\colon 0\to U_R(\sop^\infty)\to B \to C\to 0$ to the
  unique element of $\cE$ that is isomorphic to $\myfunc^0(C)$.
  Then the pushforward of $\bP^{\pm}$ under $f^{\pm}$ is equal to $\mu_{\CL}^\pm$.
\end{corollary}

Recall from Section \ref{sec:mainheuristic} that $U_R(\sop^\infty)$ is a $C_2$-module,
where for each $p\in \sop$, $C_2$ acts as the automorphism group of the
degree $2$ \'etale $\Q_p$-algebra $A_p$. Since the set $\sop$ only contains
odd primes, we have a direct sum decomposition $U_R(\sop^\infty)=U_-\oplus U_+$,
where the generator of $C_2$ acts by multiplication by $\pm1$ on $U_{\pm}$
(we caution the reader that these signs in the subscript are unrelated to
the signs in $\cG_R^{\pm}$ and in other places where the function
of the signs is to distinguish between imaginary and real quadratic fields).

We now deduce that Heuristic \ref{he:realquad} implies that in the limit,
as $F$ runs over real quadratic fields for which $\cO_F/\fm_{\f}$ is isomorphic
to $R$, the image of a fundamental unit of $\cO_F$ in $U_-$ becomes
equidistributed. Like in Section \ref{sec:functors}, given
$C\in \cG_{R,\sop}^+$, if $V_C\to \myfunc_0(C)$ is a surjective map from a free
$(\prod_{p\in S}\Q_p)$-module $V_C$ of rank $1$,
and $\Lambda_C\subset V_C$ denotes its kernel, then we get a map
$\varphi\colon \Ext^1_{C_2}(C,U_R(\sop^\infty))\to \Hom_{C_2}(\Lambda_C,U_R(\sop^\infty))$
(the r\^ole of $A$ in Secion \ref{sec:functors} is played by $U_R(\sop^\infty)$ here).
Given a $C_2$-equivariant homomorphism from $\Lambda_C$ to
$U_R(\sop^\infty)$, we may evaluate it at a generator of $\Lambda_C$ to obtain
an element of $U_R(\sop^\infty)$. Since $C_2$ acts by multiplication by $-1$ on $\myfunc_0(C)$,
the resulting element is, in fact, contained in $U_-$; and since the $C_2$-action
by $-1$ on $U_-$ is induced by a ring automorphism of $R$, the $\Aut_{\ring}(U_R(\sop^\infty))$-orbit
of the resulting element of $U_-$ is independent of the generator of $\Lambda_C$
on which we evaluate the homomorphism.

\begin{corollary}\label{cor:PFunit}
  For $C\in \cG_{R,\sop}^+$, let $\Lambda_C$ and
  $$
    \varphi\colon\! \Ext^1_{C_2}(C,U_R(\sop^\infty))\to \Hom_{C_2}(\Lambda_C,U_R(\sop^\infty))
  $$
  be as just introduced, and let $\lambda\in \Lambda_C$ be an arbitrary generator.
  Denote by $U_-/\Aut_{\ring}U_R(\sop^\infty)$ the set of $(\Aut_{\ring}U_R(\sop^\infty))$-orbits on $U_-$.
  Let $\mu$ be the pushforward of the probability measure $\bP^+$
  under the map
  \begin{eqnarray*}
    \cG_{R,\sop}^+ & \to & U_-/\Aut_{\ring}U_R(\sop^\infty)\\
    \Theta & \mapsto & \varphi(\Theta)(\lambda).
  \end{eqnarray*}
  Then for every $O\in U_-/\Aut_{\ring}U_R(\sop^\infty)$
  we have
  $$
  \mu(O) = \frac{\#O}{\#U_-}.
  $$
\end{corollary}

Next, we deduce that Heuristic \ref{he:realquad}
implies that as $F$ runs over real quadratic fields with the property
that $\cO_F/\fm_{\f}$ is isomorphic to $R$, the distribution of
$\Z_{\sop}\otimes\Cl_F$ and of the image of the fundamental unit of
$\cO_F$ in $((\cO_F/\fm_{\f})^\times/\{\pm 1\})_{\sop}\cong U_R(\sop^\infty)$ are
independent of each other.
\begin{corollary}
The pushforward of the probability measure $\bP^+$ under the map
$\cG_R^+\to \cE\times (U_-/\Aut_{\ring}U_R(\sop^\infty))$ that is the product
of the maps of Corollaries \ref{cor:PFclassgroup} and \ref{cor:PFunit}
is the product measure of $\mu_{\CL}^-$ and of the uniform measure
as in Corollary \ref{cor:PFunit}.
\end{corollary}

Finally, we deduce that if Heuristic \ref{he:realquad} holds, then in the limit as $F$
runs over all real quadratic fields for which $\Z_{\sop}\otimes \Cl_F$ is isomorphic
to a given $\sop$-group $G$ and the $\sop^\infty$-torsion of the left hand term of sequence \eqref{eq:SeqRay} is
isomorphic to a given quotient $Q$ of $U_R(\sop^\infty)$, the sequence \eqref{eq:SeqRay}
equidistributes over $\Ext^1_{C_2}(G,Q)$.

For $C\in \cG_{\sop}^+$, let $\cG^+(C)\subset \cG^+_{R,\sop}$ be the subset
  consisting of all extensions $\Theta$ of $C$ by $U_R(\sop^\infty)$.
  Let $\Z_S(-1)$ be a $\Z_{\sop}[C_2]$ module that is free of rank $1$ over $\Z_{\sop}$, and
  with the generator of $C_2$ acting by $-1$.
  Once one fixes an isomorphism between $\Lambda_C$ and $\Z_{\sop}$,
  every element $\Theta$ of $\cG^+(C)$ defines a $C_2$-equivariant map
  $\varphi(\Theta)\colon \Z_{\sop}(-1)\to U_R(\sop^\infty)$ and an extension $\myfunc^0(\Theta)$ of
  $\myfunc^0(C)$ by $U_R(\sop^\infty)/\im \varphi(\Theta)$, as explained in Section \ref{sec:functors}.
  Let $\pairs(C,U_R(\sop^\infty))$ be the set of such pairs $(\psi,\phi)$ with
  $\psi\in \Hom_{C_2}(\Z_{\sop}(-1),U_R(\sop^\infty))$ and $\phi\in \Ext^1_{C_2}(C,U_R(\sop^\infty)/\im\psi)$.
  For $(\psi,\phi)\in \pairs(C,U_R(\sop^\infty))$, let $O(\psi,\phi)$ be the orbit
  of that element under the natural action of $\Aut_{\ring}U_R(\sop^\infty)\times \Aut C$ on
  $\pairs(C,U_R(\sop^\infty))$.
\begin{corollary}
  Retain the notation just introduced.
  Let $C\in \cG^+_{R,\sop}$, and let $\bP^+(C)$ be the
  probability measure on $\cG^+(C)$ obtained by restricting and renormalising $\bP^+$.
  Let $\mu$ be the pushforward of $\bP^+(C)$ under the map 
  \begin{eqnarray*}
    \cG^+(C) & \to & \pairs(C,U_R(\sop^\infty)),\\
    \Theta & \mapsto & (\varphi(\Theta),\myfunc^0(\Theta)).
  \end{eqnarray*}
  Then for all $(\psi,\phi)\in \pairs(C,U_R(\sop^\infty))$, we have
  %$(\Aut_{\ring}U_R(\sop^\infty)\times \Aut C)$-orbit $O$ on $\pairs(C,U_R(\sop^\infty))$,
  $$
  \mu(O(\psi,\phi))=\frac{\#O(\psi,\phi)}{\#U_-\cdot\#\Ext^1_{C_2}(C,U_R(\sop^\infty)/\im\psi)}.
  $$
\end{corollary}

To complete the discussion, we record the immediate consequence of Proposition \ref{prop:sizeaut}
that Heuristic \ref{he:realquad} implies the heuristic of the second author and Sofos on ray
class groups of imaginary quadratic fields \cite{PS}.

\begin{corollary}\label{cor:equidistrExt}
  For $C\in \cG_{\sop}^-$, let $\cG^-(C)\subset \cG_R^-$ be the subset consisting of all
  extensions of $C$ by $U_R(\sop^\infty)$. Then for every $(\Aut_{\ring}U_R(\sop^\infty)\times \Aut C)$-orbit
  $O\subset \Ext^1_{C_2}(C,U_R(\sop^\infty))$, the probability measure $\mu$ on $\Ext^1_{C_2}(C,U_R(\sop^\infty))$
  obtained by restricting and renormalising $\bP^-$ satisfies
  $$
    \mu(O)=\#O/\#\Ext^1_{C_2}(C,U_R(\sop^\infty)).
  $$
\end{corollary}

\subsection{Consequences for $\ell$-torsion subgroups}
In this subsection we derive from Heuristic \ref{he:realquad}
predictions about the average $\ell$-torsion of ray class groups of
real quadratic fields, and thereby show that the heuristic is compatible
with \cite[Theorem 1]{Varma}.

If $\Theta\colon 0\to A\to B\to C\to 0$ is a short exact sequence of abelian
groups, and $\ell$ is a prime number, then by applying the functor $\Hom(\Z/\ell\Z,\bullet)$
one obtains the exact sequence
$$
0\to A[\ell] \to B[\ell]\to C[\ell] \stackrel{\delta_{\ell}(\Theta)}{\to} A/\ell A.
$$
%where $\delta_{\ell}$ is defined as follows: gives $c\in C[\ell]$, let $b\in B$
%be any lift of $c$. Then $\ell b$ maps to $0$ in $C$, so belongs to $A$.
%If $b'$ is any other lift of $c$ in $b$, then $\ell b'$ differs from $\ell b$
%by an element of $\ell A$, so that $c\mapsto \ell b$ is a well-defined map
%$C[\ell]\to A/\ell A$. This is the map $\delta_{\ell}(\Theta)$. 
\begin{lemma}\label{lem:deltaell}
  Let $\ell$ be an odd prime, and let $A$, $C$ be compact $\Z[C_2]$-modules.
  Then the map $\Ext^1_{C_2}(C,A)\to \Hom_{C_2}(C[\ell],A/\ell A)$ given
  by $\Theta\mapsto \delta_{\ell}(\Theta)$ is a surjective group homomorphism.
\end{lemma}
\begin{proof}
  The fact that the map is a group homomorphism can be shown by an explicit Baer
  sum calculation, which we leave to the reader.

  We now prove surjectivity. We have $\Ext^1_{C_2}(C,A) = \Ext^1_{C_2}(A^\lor,C^\lor)$
  and $\Hom_{C_2}(C[\ell],A/\ell A)=\Hom_{C_2}(A^\lor[\ell],C^\lor/\ell C^\lor)$,
  and the map $\Theta^\lor \mapsto \delta_{\ell}(\Theta^\lor)$ is, under these identifications,
  the same map. Thus, we may equivalently assume that $A$ and $C$ are finitely generated
  rather than compact. Next, the map clearly factors through 
  $$
  \Ext^1_{C_2}(C,A) \to \Ext^1_{\Z_{\ell}[C_2]}(\Z_{\ell}\otimes C,\Z_{\ell}\otimes A),
  $$
  and since $\Z_{\ell}[C_2]$ is flat over $\Z[C_2]$, that latter map is surjective.
  It therefore suffices take $A$ and $C$ to be finitely generated $\Z_{\ell}[C_2]$-modules.
  Moreover, we have $\Z_{\ell}[C_2]\cong \Z_{\ell}\times \Z_{\ell}$,
  so as in earlier proofs, the claim easily reduces to the analogous claim in the
  category of finitely generated $\Z_{\ell}$-modules. For a final reduction, since $C$
  is a direct sum of its torsion subgroup and a free group, we may assume, without
  loss of generality, that $C$ is torsion.

  Let $f\in \Hom(C[\ell],A/\ell A)$, and suppose that $A$ has presentation $\langle \cG_A | \cR_A\rangle$
  and $C$ has presentation $C=\Z^{\cG_C} / \langle \ell^{e_c}c : c \in \cG_C\rangle$ with $e_c\in \Z_{\geq 1}$ for
  all $c\in \cG_C$. For each $c\in \cG_C$, fix an arbitrary lift of $f(\ell^{e_c-1}c)\in A/\ell A$ to
  $A$, and abusing notation, temporarily also denote it by $f(\ell^{e_c-1}c)$.
  Define 
  $$
    B=\langle \cG_A, \cG_C | \cR_A, \ell^{e_c}c - f(\ell^{e_c-1}c)\text{ for all }c\in \cG_C\rangle.
  $$
  Then an easy computation shows that $B$ is an extension of $C$ by $A$ whose class in
  $\Ext^1(C,A)$ is mapped to $f$ under the map in the statement of the lemma.
\end{proof}

For the purposes of the rest of the section, suppose that $S=\{\ell\}$, where
$\ell$ is an odd prime.

Suppose that $F$ is a real quadratic field. Then applying the above construction
to the exact sequence $S_F^{\Ara}(\fm)$,
we obtain a map
$$
\delta_{F,\ell}^{\Ara}\colon \Pic^0_F[\ell] \to (\cO_F/\fm_{\f})^\times/((\cO_F/\fm_{\f})^\times)^{\ell}\cong U_R/U_R^{\ell},
$$
where, as in the introduction, the final isomorphism is obtained by means of an
identification $r\colon R\stackrel{\sim}{\to} \cO_F/\fm_{\f}$.
Similarly, applying the construction to the exact sequence $S_F^{\fin}(\fm)$, we obtain a map
$$
\delta_{F,\ell}^{\fin}\colon \Cl_F[\ell] \to \frac{(\cO_F/\fm_{\f})^\times}{((\cO_F/\fm_{\f})^\times)^{\ell}(\cO_F^\times)_{\mymod \fm_{\f}}}.
$$
Finally, when applying it to the exact sequence
$$
0 \to \bT_F \to \Pic_F^0 \to \Cl_F \to 0,
$$
we see that the corresponding sequence of $\ell$-torsion subgroups is exact, since
$\bT_F$ is divisible.
Recall that we have $\bT_F=\cO_F^\times\otimes_{\Z}\R/\Z$,
so that $\bT_F[\ell] = \cO_F^\times \otimes \tfrac{1}{\ell}\Z/\Z$, where we
identify, for each $u\in \cO^\times_F$, the element $u\otimes \tfrac{1}{\ell}$ with
the class of $(\cO_F,\tfrac{1}{\ell}\cdot u)\in \Id_F\times_{\R_{>0}}\overline{F_{\R}^\times}$
in $\Pic_F^0$. Under these identifications, we have that for every $u\in\cO_F^\times$,
the element $\delta_{F,\ell}^{\Ara}(u\otimes \tfrac{1}{\ell})$ is just the image
of $u$ under the quotient map $\cO_F^\times\mapsto (\cO_F/\fm_{\f})^\times/((\cO_F/\fm_{\f})^\times)^{\ell}$.
Temporarily denote that image by $\bar{u}_{\ell}$. Then we obtain a commutative diagram with exact rows
$$
\xymatrix{
  0 \ar[r] & \bT_F[\ell] \ar[d]\ar[r] & \Pic_F^0[\ell] \ar[d]^{\delta_{F,\ell}^{\Ara}}\ar[r] & \Cl_F[\ell] \ar[d]^{\delta_{F,\ell}^{\fin}}\ar[r] & 0\\
  0 \ar[r] & \langle \bar{u}_{\ell}: u \in \cO_F^\times\rangle \ar[r] & \frac{(\cO_F/\fm_{\f})^\times}{((\cO_F/\fm_{\f})^\times)^{\ell}} \ar[r] &
  \frac{(\cO_F/\fm_{\f})^\times}{((\cO_F/\fm_{\f})^\times)^{\ell}\langle \bar{u}_{\ell}: u \in \cO_F^\times\rangle} \ar[r] & 0.
}
$$

For $N\in \Z_{\geq 0}\cup\{\infty\}$, define
$$
\eta_N(\ell) = \prod_{i=1}^N (1-\ell^{-i}).
$$
For two groups $A$, $B$, let $\Surj(A,B)$ denote the set of surjective
homomorphisms $A\to B$.
As in the previous subsection, decompose $U_R/U_R^{\ell}$ as
$U_R/U_R^{\ell}=U_+ \oplus U_-$,
where the generator of $C_2$ acts by multiplication by $\pm 1$ on
$U_{\pm}$. For $X\in\R_{>0}$, let $\cF_X^+(R)$ be as in the introduction.
For $(F,r)\in \cF_X^+(R)$, we identify $\frac{(\cO_F/\fm_{\f})^\times}{((\cO_F/\fm_{\f})^\times)^{\ell}}$
with $U_R/U_R^{\ell}$ via $r$, so that $\delta_{F,\ell}^{\Ara}(\Pic^0_F[\ell])$
is a subspace of $U_-$.
As an immediate consequence of Theorem \ref{thm:realquad}
and of the calculations in \cite[\S 9]{CL}, we now deduce the following predictions.
\begin{corollary}\label{cor:elltorsion1}
  Suppose that Heuristic \ref{he:realquad} holds for the function
  $f\colon \cG_{R,S}^+\to \Z$ that sends an exact sequence
  $0\to U_R(\ell^\infty)\to B\to C\to 0$ to $\dim_{\F_{\ell}}B[\ell]$. Then for every $j\in \Z_{\geq 1}$
  and for every subspace $W \subset U_-$, in the limit as
  $X\to \infty$, the proportion among $(F,r)\in \cF_X^+(R)$ of those satisfying
  $$
    \dim_{\F_{\ell}}(\Pic_F^0[\ell])=j\quad\text{ and }\quad \delta_{F,\ell}^{\Ara}(\Pic^0_F[\ell])=W
  $$
  is equal to
  $$
  \frac{\eta_{\infty}(\ell)}{\ell^{(j-1)j}\eta_{j-1}(\ell)\eta_j(\ell)}\cdot\frac{\#\Surj(\F_{\ell}^j,W)}{\#\Hom(\F_{\ell}^j,U_-)}.
  $$
\end{corollary}

\begin{corollary}
  Suppose that Heuristic \ref{he:realquad} holds for the same function
  $f$ as in Corollary \ref{cor:elltorsion1}. Let
  $W\subset U_-$ be a $1$-dimensional subspace. Then:
  \begin{enumerate}
    \item in the limit as $X\to \infty$, the proportion among $(F,r)\in \cF_X^+(R)$
      of those for which $\delta_{F,\ell}^{\Ara}(\bT_F[\ell]) = W$ is equal to
      $\frac{(\ell-1)}{\#U_-}$.
    \item for every $j\in \Z_{\geq 1}$, the proportion among $(F,r)\in \cF_X^+(R)$
      of those satisfying
  $$
    \dim_{\F_{\ell}}(\Pic_F^0[\ell])=j\quad\text{ and }\quad \delta_{F,\ell}^{\Ara}(\bT_F[\ell])=W
  $$
  is equal to
  $$
  \frac{\eta_{\infty}(\ell)}{\ell^{(j-1)j}\eta_{j-1}(\ell)\eta_j(\ell)}\cdot\frac{\ell-1}{\#U_-}.
  $$
  \end{enumerate}
\end{corollary}

The rest of the section is devoted to the proof of the following result,
which in particular shows that Heuristic \ref{he:realquad} is compatible
with Varma's result on the average of $3$-torsion of ray class groups of
real quadratic fields (see \cite{PS} for the corresponding result for
imaginary quadratic fields).

\begin{proposition}\label{prop:Varma}
  Let $\ell$ be an odd prime, let $\cP_1$, respectively $\cP_{\pm 1}$
  be the set of prime numbers $p|\fm_{\f}$ that satisfy $p\equiv 1\pmod\ell$,
  respectively $p\equiv \pm 1\pmod \ell$,
  and for $X\in \R_{>0}$ let $\cF_X^+$ be the set of real quadratic fields of
  conductor less than $X$. Suppose that Heuristic \ref{he:realquad}
  holds for the function $f$ defined in Corollary \ref{cor:elltorsion1}.
  Then the limit
  $$
  \Av^+(\ell)=\lim_{X\to \infty} \frac{\sum_{F\in \cF_X^+}\#\Cl_F(\fm)[\ell]}{\#\cF_X^+}
  $$
  exists, and is equal to:
  \begin{enumerate}
    \item $\ell^{\#\cP_1}\left(1+\frac{1}{\ell}\prod_{p\in \cP_{\pm 1}}\frac{p(\ell+1)+2}{2(p+1)}\right)$
      if $\ell\nmid \fm_{\f}$;
    \item $\ell^{\#\cP_1}\left(1+\frac{2}{\ell+1}\prod_{p\in \cP_{\pm 1}}\frac{p(\ell+1)+2}{2(p+1)}\right)$
      if $\ell | \fm_{\f}$ but $\ell^2\nmid \fm_{\f}$;
    \item $\ell^{\#\cP_1+1}\left(1+\prod_{p\in \cP_{\pm 1}}\frac{p(\ell+1)+2}{2(p+1)}\right)$
      if $\ell>3$ and $\ell^2 | \fm_{\f}$;
    \item $3^{\#\cP_1+1}\left(1+\frac{5}{4}\cdot\prod_{p\in \cP_{\pm 1}}\frac{2p+1}{p+1}\right)$
      if $\ell=3$ and $\ell^2|\fm_{\f}$.
  \end{enumerate}
\end{proposition}
%The proof of the proposition will occupy the rest of the subsection.
\begin{lemma}\label{lem:elltorsexp}
  Let $f\colon \cG_{R,\sop}^+\to \Z$ be the function that sends
  an exact sequence $0\to U_R(\ell^\infty)\to B \to C\to 0$ to $\#B[\ell]$.
  Then one has
  $$
    \bE^+(f) = \ell\cdot \#U_+ +
    \#\tfrac{U_R}{U_R^\ell}.
  $$
\end{lemma}
\begin{proof}
  The proof is essentially identical to that of \cite[Prop. 2.11]{PS}, when
  phrased in terms of Arakelov ray class groups. We will
  reproduce it here for the convenience of the reader.

  For $\Theta\in \cG_{R,S}^+$, write $B(\Theta)$ for the middle term of the sequence.
  %and $C(\Theta)$ for the right-most term, i.e. for $B(\Theta)/U_R$.
  By definition, we have
  $$
    \bE^+(f) = \sum_{\Theta\in \cG_{R,\sop}^+} \frac{(c_R^+)^{-1}\#B(\Theta)[\ell]}{\#\Aut_{\ring}\Theta}.
  $$
  By Proposition \ref{prop:sizeaut} and Corollary \ref{cor:PFclassgroup}, this is
  equal to
  $$
  \sum_{C\in \cG_{\sop}^+}\left( \frac{\mu^+_{\CL}(\myfunc^0(C))}{\#\Ext^1_{C_2}(C,U_R/U_R^\ell)}\sum_{\Theta\in \Ext^1_{C_2}(C,U_R/U_R^\ell)}\#B(\Theta)[\ell]\right).
  $$
  By Lemma \ref{lem:deltaell}, this equals
  \begin{equation}\label{eq:bE1}
    \quad\quad\sum_{C\in \cG_{\sop}^+}\left(\frac{\mu^+_{\CL}(\myfunc^0(C))}{\#\Hom(C[\ell],U_-)}\sum_{\delta\in \Hom(C[\ell],U_-)}\frac{\#U_R[\ell]\#C[\ell]}{\#\delta(C[\ell])}\right).
  \end{equation}
  For each $\chi\in (U_-)^\lor$ and for each $\delta\in \Hom(C[\ell],U_-)$,
  define $\triv_{\chi\circ \delta}$ to be $1$ if $\chi$ vanishes on the image of $\delta$,
  and $0$ otherwise. Then for each such $\delta$, we have
  $$
  \#U_-\big/\#\delta(C[\ell])=\#\{\chi\in (U_-)^\lor : \triv_{\chi\circ\delta}=1\},
  $$
  so that we may rewrite \eqref{eq:bE1} as 
  $$
  \sum_{C\in \cG_{\sop}^+}\left(\frac{\mu^+_{\CL}(\myfunc^0(C))}{\#\Hom(C[\ell],U_-)}\sum_{\delta}\#U_+\#C[\ell]
  \sum_{\chi}\triv_{\chi\circ\delta}\right).
  $$
  Here and below, the sum over $\delta$ runs over $\delta\in \Hom(C[\ell],U_-)$,
  and the sum over $\chi$ runs over $\chi\in (U_-)^\lor$.
  Exchanging the order of summation, we may rewrite this as
  \begin{equation}\label{eq:bE2}
    \sum_{C\in \cG_{\sop}^+}\left(\mu^+_{\CL}(\myfunc^0(C))\#U_+\#C[\ell]\sum_{\chi}
    \frac{\sum_{\delta}\triv_{\chi\circ\delta}}{\#\Hom(C[\ell],U_-)}\right).
  \end{equation}
  The summand in the inner sum corresponding to $\chi\in (U_-)^\lor$ is
  equal to $1$ if $\chi$ is trivial, and to $1/\#C[\ell]$ otherwise. Thus \eqref{eq:bE2}
  simplifies to
  $$
  \sum_{C\in \cG_{\sop}^+}\left(\mu^+_{\CL}(\myfunc^0(C))\#U_+\#C[\ell]\Big(1+\frac{\#U_- - 1}{\#C[\ell]}\Big)\right).
  $$
  For all $C\in \cG_{\sop}^+$ we have $\#C[\ell] = \ell\cdot\#\myfunc^0(C)[\ell]$, and by \cite[Example 5.12]{CL} we have
  $\sum_{C\in \cG_{\sop}^+}(\mu^+_{\CL}(\myfunc^0(C))\#\myfunc^0(C)[\ell]) = 1+\tfrac{1}{\ell}$,
  so that the above expression for $\bE^+(f)$ becomes
  \begin{equation*}
    \begin{split}
    \#U_+\cdot(\ell +1 + \#U_- - 1) & = \#U_+\cdot(\ell+\#U_-)\\
                                                               &= \ell\cdot \#U_+ +
    \#\left(U_R/U_R^\ell\right),
  \end{split}
\end{equation*}
  as claimed.
\end{proof}

\begin{proposition}\label{prop:VarmaFixedR}
  Suppose that Heuristic \ref{he:realquad} holds for the function $f$ defined in
  Corollary \ref{cor:elltorsion1}. Then one has
  $$
   \lim_{X\to \infty}\frac{\sum_{(F,r)\in \cF_X^+(R)} \#\Cl_F(\fm)[\ell]}{\#\cF_X^+(R)} = \#U_+ +\frac{\#\left(U_R/U_R^\ell\right)}{\ell}.
  $$
\end{proposition}
\begin{proof}
  If $F$ is a real quadratic field, then one has $\#\Cl_F(\fm)[\ell]=\frac{\#\Pic^0_F(\fm)[\ell]}{\ell}$,
  so the claim follows immediately from Lemma \ref{lem:elltorsexp}.
\end{proof}

\begin{proof}[Proof of Proposition \ref{prop:Varma}]
  Denote the limit in Proposition \ref{prop:VarmaFixedR} by $\Av_R^+(\ell)$.
  Let $p_R$ be the limit as $X\to \infty$ of the proportion among all $F\in \cF_X^+$
  of those for which $\cO_F/\fm_{\f}$ is isomorphic to $R$. Then we have
  \begin{eqnarray*}%\label{eq:twosums}
    \Av^+(\ell) = \sum_R p_R\cdot\Av_R^+(\ell)=\sum_Rp_R\cdot\#U_+ +\frac{1}{\ell}\sum_Rp_R\cdot\#\left(U_R/U_R^\ell\right),
  \end{eqnarray*}
  where the sums run over all possible rings $R$. To prove the proposition,
  we will explicitly evaluate the two sums on the right hand side.

  The first sum is easy to evaluate: $\#U_+$ is independent of
  $R$, namely we have
  $$
    \#U_+ = \leftchoice{\ell^{\#\cP_1}}{\ell^2\nmid \fm_{\f}}{\ell^{\#\cP_1+1}}{\ell^2|\fm_{\f}}.
  $$
  We will now evaluate the second sum, which we denote by $s(\fm_{\f})$.
  If $\fm_{\f}=\fn\fn'$, where $\fn$ and $\fn'$ are coprime, then as $F$
  varies over real quadratic fields, the structures of $\cO_F/\fn$ and
  $\cO_F/\fn'$ are, in the limit, independent of each other, whence it easily
  follows that one has $s(\fm_{\f}) = \prod_{p^k\parallel \fm_{\f}}s(p^k)$,
  with the product running over the distinct prime divisors of $\fm_{\f}$. It
  therefore suffices to evaluate $s(\fm_{\f})$ in the special case that
  $\fm_{\f}=p^k$, where $p$ is a prime number and $k\in \Z_{\geq 1}$, which we
  will now do.

  Recall that we have $R=\cO_A/(p^k)$, where $\cO_A$ is the integral closure of
  $\Z_p$ in a degree $2$ \'etale $\Q_p$-algebra $A_p$. If $p\not \in \{-1,0,1\} \pmod \ell$,
  then we have $\#(U_R/U_R^\ell)=1$. The calculation now breaks up into several cases.\vspace{1em}
  \newline
  \textit{Case 1: $p\equiv 1\pmod \ell$.}
  Write $A_p=\Q_p[x]/(x^2-d)$, where $d\in \Z_{>0}$ is square-free. One has
  $A_p \cong \Q_p\times \Q_p$ if and only if $p\nmid d$ and $d$ is a square
  modulo $p$; $A_p$ is an unramified quadratic field extension of $\Q_p$ if and
  only if $p\nmid d$ and $d$ is a non-square modulo $p$; and $A_p$ is
  a quadratic ramified extension of $\Q_p$ if and only if $p|d$. In these
  three cases, we have that $p_R$ is equal to $\frac{p}{2(p+1)}$, respectively
  $\frac{p}{2(p+1)}$, respectively $\frac{1}{p+1}$; and $\#(U_R/U_R^\ell)$
  is equal to $\ell^2$, respectively $\ell$, respectively $\ell$. We deduce
  that in this case we have
  $$
    s(p^k) = \left(\frac{\ell^2\cdot p}{2(p+1)} + \frac{\ell\cdot p}{2(p+1)} + \frac{\ell}{p+1}\right)
    = \ell\cdot\frac{p(\ell+1)+2}{2(p+1)}.
  $$
  \newline
  \textit{Case 2: $p\equiv -1\pmod \ell$.}
  Considering the same three cases as above, we have the same proportions $p_R$,
  and this time $\#(U_R/U_R^\ell)$ is equal to $1$ in the split
  case, to $\ell$ in the case of $A_p/\Q_p$ being an unramified quadratic field extension,
  and to $1$ in the ramified case. In summary, we have
  $$
    s(p^k) = \frac{p}{2(p+1)} + \frac{\ell\cdot p}{2(p+1)} + \frac{1}{p+1} = \frac{p(\ell+1) + 2}{2(p+1)}.
  $$
  \newline
  \textit{Case 3: $p=\ell$, $k=1$.}
  Again, we have the same three possibilities for $R$ with proportions $p_R$
  equal to $\frac{\ell}{2(\ell+1)}$, $\frac{\ell}{2(\ell+1)}$, and $\frac{1}{\ell+1}$,
  respectively. In these three cases, $\#(U_R/U_R^\ell)$ is equal to $1$, $1$, and
  $\ell$, respectively, so that one has
  $$
  s(\ell) = \frac{\ell}{2(\ell+1)} + \frac{\ell}{2(\ell+1)} + \frac{\ell}{\ell+1} = \frac{2\ell}{\ell+1}.
  $$
  \newline
  \textit{Case 4: $p=\ell>3$, $k>1$.}
  In this case, one has $\#(U_R/U_R^\ell)=\ell^2$ for all possible rings $R$,
  so that we have
  $$
    s(\ell^k) = \ell^2.
  $$
  \newline
  \textit{Case 5: $p=\ell=3$, $k>1$.}
  In this case, one has $\#(U_R/U_R^\ell)=9$, unless $A_3 \cong \Q_3(\zeta_3)$. For
  that latter ring, which occurs with frequency $p_R=1/8$, one has $\#(U_R/U_R^\ell)=27$.
  In summary, in this case we have
  $$
    s(3^k) = \frac{9\cdot7}{8} + \frac{27}{8} = \frac{45}{4}.
  $$
  This completes the computation of $\Av^+(\ell)$.
\end{proof}

%\begin{proof}
%  Recall the definition of the ring $R$ that appears in the statement of Heuristic
%  \ref{he:realquad}: for every prime divisor $p$ of $\fm_{\f}$ we choose a quadratic
%  \'etale $\Q_p$ algebra $A_p$, we let $\cO_A$ be the integral closure of
%  $\prod_{p|\fm_{\f}} \Z_p$ in $A=\prod_{p|\fm_{\f}}A_p$, and we set $R=\cO_A/\fm_{\f}$.
%  For every such choice of quadratic algebras, let
%  $$
%  \Av^+_R(\ell) = \lim_{X\to \infty} \frac{\sum_{(F,r)\in \cF_X^+(R)}}{}
%  $$
%  that is the subject of Heuristic \ref{he:realquad} 
%\end{proof}

%\section{Heuristics at the ``bad'' prime}

\end{document}